\documentclass[12pt]{article}
\usepackage{graphicx}
\usepackage{color}
\usepackage{amsmath}
\usepackage{amssymb}
\usepackage{amscd}
\usepackage{amsthm}
\usepackage{amsopn}
\usepackage{xspace}
\usepackage{verbatim}
\usepackage{amsmath}
\usepackage{amsfonts}
\usepackage{epic}
\usepackage{eepic}
\usepackage{stmaryrd}
\usepackage{empheq}
\usepackage[active]{srcltx}
\usepackage[a4paper,margin=3cm]{geometry}
\definecolor{red}{rgb}{1,0,0}
\definecolor{green}{rgb}{0,1,0}
\definecolor{SeaGreen}{RGB}{46,139,87}
\definecolor{Maroon}{RGB}{128,0,0}

\newcommand{\N}{\mathbb{N}}

\newcommand{\C}{{\mathbb{C}}}
\newcommand{\R}{{\mathbb{R}}}

\newcommand{\A}{{\mathcal A}}
\newcommand{\B}{\mathcal B}

\newcommand{\D}{\mathcal D}
\def\Dg {{\mathfrak D}}

\def\Kg {{\mathcal K}}

\newcommand{\LL}{\mathcal L}

\newcommand{\OO}{\mathcal O}

\def\Rg {{\mathcal R}}
\def\Sg {{\mathcal S}}

\def\Tg {{\mathcal T}}

\def\curl{\text{\rm curl}}

\def\curl{\text{\rm curl\,}}

\def\Div{\text{\rm div\,}}

\renewcommand {\Re}{{\rm Re\,}}
\renewcommand{\Im}{{\rm Im\,}}

\def\Ai{\text{\rm Ai\,}}

\def\0{\mathbf  0}

\def\XXint#1#2#3{{\setbox0=\hbox{$#1{#2#3}{\int}$ }
\vcenter{\hbox{$#2#3$ }}\kern-.6\wd0}}

 

\errorcontextlines=0 \numberwithin{equation}{section}
\theoremstyle{plain}
\newtheorem{theorem}{Theorem}[section]
\newtheorem{lemma}[theorem]{Lemma}

\newtheorem{proposition}[theorem]{Proposition}

\newtheorem{remark}[theorem]{Remark}
\newtheorem{corollary}[theorem]{Corollary}

\title{Stability of laminar monotone shear flows in a channel for high
Reynolds number}
\author{ Y. Almog, Department of
  Mathematics, \\ Braude College of Engineering, \\ 
    Carmiel 2161002, Israel \\~\\
  and \\~\\
\noindent   B. Helffer, Laboratoire de Math\'ematiques Jean Leray, \\CNRS and  Nantes Universit\'e, \\
 44000 Nantes  Cedex France}

\begin{document}
\maketitle
\bibliographystyle{siam}
\begin{abstract} 
  We consider the stability of a laminar flow $U\in C^4([-1,1])$ in the
  two-dimensional channel $\R\times[-1,1]$ in the large Reynolds number
  limit.  Assuming that $U$ is strictly monotone but allowing $U''$ to
  vanish, we obtain that if the operator
  \begin{displaymath}
     \Kg_{\nu}=-\frac{d^2}{dx^2}+\frac{U^{\prime\prime}}{U-\nu} \,,
  \end{displaymath}
is strictly positive for all $\nu\in\R$ for which $U^{\prime\prime}(U^{-1}(\nu))=0$,
then $U$ is stable for sufficiently large Reynolds number. This
contribution generalizes previous results  mostly by allowing long wave
perturbations (but much shorter than the Reynolds number). 
\end{abstract}

\section{Introduction}
 Consider the incompressible Navier-Stokes equations in the
   two-dimensional pipe  $D=\R\times (-1,1)$
   \begin{equation}
   \label{eq:1}
   \begin{cases}
   \partial_t {\mathbf v} - \epsilon \Delta {\mathbf v} + {\mathbf v} \cdot \nabla {\mathbf v} = -
   \nabla p & \text{in } \R_+\times D  \\
   {\mathbf v}=v_b\; \hat{i}_1  & \text{on } \R_+\times\partial D  \,,
   \end{cases} 
   \end{equation}
   where  $\hat{i}_1 = (1,0)$, ${\mathbf v}=(v_1,v_2)$ is the fluid
   velocity, and $p$ is the pressure. \\
   The parameter
    \begin{equation} \label{defreynolds}
    R :=\frac 1 \epsilon
    \end{equation}  is the Reynolds
    number of the  flow and 
    \begin{displaymath}
      v_b:\partial D\to\R
    \end{displaymath}
is the boundary velocity.\\ Since the
    flow is incompressible we must have
   \begin{displaymath}
     \Div {\mathbf v}=0\,.
   \end{displaymath}
   We linearize \eqref{eq:1} near the laminar flow (cf. \cite{AH1}) 
   \begin{displaymath}
     {\mathbf v}=U(x_2)\hat{i}_1 \,,
   \end{displaymath}
   to obtain the linearized equation 
\begin{displaymath}
  {\bf u}_t-\mathcal T_0({\bf u},q)=0
\end{displaymath}
   where  ${\bf u}=(u_1,u_2)$ and $q$  are defined on $\mathbb R_+\times
   D$, and $\mathcal T_0$ is the map 
     \begin{equation}
   \label{eq:2}
   ({\bf u},q) \mapsto  {\mathcal T}_0 ({\bf u} , q ):=  -  \epsilon \,\Delta {\mathbf u} + U\, \frac{\partial{\mathbf
       u}}{\partial x_1}+ \, u_2\, U^\prime\, \hat{i}_1 - \nabla q\,.
   \end{equation}
   
We proceed with a formal derivation of the Orr-Sommerfeld equation,
   intentionally skipping the definitions of ${\bf v}$, $p$, ${\bf u}$, ${\bf f}$, 
   and $q$. Interested readers can read the entire derivation in
   \cite{AH1}. The associated resolvent equation for
 $\mathcal T_0$ assumes the form
   \begin{equation}
   \label{eq:3}
     \Tg_0({\bf u},q)-\Lambda{\bf u}={\bf f} \,,
   \end{equation}
   where $\Div {\bf u}=0$ and  $\Lambda\in\C$ is the spectral parameter.\\ 
    Hence, we may define a stream function
   \begin{displaymath}
     {\mathbf u}=\nabla_\perp\psi=(-\psi_{x_2},\psi_{x_1}) \,.
   \end{displaymath}
  Substituting the above into \eqref{eq:3} and then taking the curl of
  the ensuing equation for $\psi$ yields
  \begin{equation}
  \label{eq:4}
    \Big(-  \epsilon \Delta^2 + U\frac{\partial}{\partial x_1}\Delta -
    U^{\prime\prime}\frac{\partial}{\partial x_1}  - \Lambda \,  \Delta\Big)\psi=F \,,
  \end{equation}
where $F=\curl {\bf f}$. 

  We consider $U\in C^4([-1,1])$ satisfying 
  \begin{equation}
  \label{eq:5}
     |U^\prime(x)|\geq {\bf m} >0
    \end{equation}
  Substituting $\psi(x_1,x_2)=\phi(x_2) \, e^{i\alpha x_1}$ into
  \eqref{eq:4} with $ \phi:(-1,1)\to\C$ yields  the equation
  \begin{subequations}\label{eq:1.7}
  \begin{equation}
    \B_{\lambda,\alpha,\beta}\,\phi=f \,, 
  \end{equation}
  where (setting $x_2=x$)
   \begin{equation}
\label{eq:6}
     \B_{\lambda,\alpha,\beta} =(\LL_\beta -\beta\lambda)\Big(\frac{d^2}{dx^2}-\alpha^2\Big)  -i\beta U^{\prime\prime} \,,
   \end{equation}
   where
  \begin{equation}
  \label{eq:p29}
    \LL_\beta  = -\frac{d^2}{dx^2}+i\beta U\,.
  \end{equation}
  \end{subequations}
In the above 
\begin{equation}\label{defbeta}
\beta = \alpha \epsilon^{-1} =\alpha R
\end{equation}
 ($R$ being the Reynolds
     number introduced in \eqref{defreynolds}), and, for $\beta \neq 0$, 
     \begin{displaymath}
     \lambda = \hat{\Lambda} -\alpha^2\beta^{-1}\,,
     \end{displaymath}
where
\begin{equation}
   \hat{\Lambda} =\frac{\Lambda}{\alpha}
\end{equation}

     We refer to Section 3 in \cite{AH1} for the
     details of the derivation.  We use the pair of parameters
     $(\alpha,\beta)$ instead of $(\alpha,R)$ since the asymptotic limit we
     consider in the sequel is
     $\beta\to\infty$. \\
 We consider the Orr-Sommerfeld operator, i.e. the Dirichlet
 realization $\B_{\lambda,\alpha,\beta}^{\mathcal D}$ of $ \B_{\lambda,\alpha,\beta}$ on the
 following domain 
  \begin{equation}
\label{eq:7}
  D(\B_{\lambda,\alpha,\beta}^{\mathcal D})=\{u\in H^4(-1,1)\,,\, u(1)=u^\prime(1)= u(-1) =u^\prime (-1) =0 \}\,.
  \end{equation}
Since $\B_{\lambda,\alpha,\beta}=\B_{\lambda,-\alpha,\beta}$ we consider the case $\alpha\geq0\,$ only in
the sequel.\\
For any $\nu\in[U(-1),U(1)]$ we define $x_\nu\in[-1,1]$ by  $$U(x_\nu)=\nu\,.$$
Notice that $x_\nu$ is unique by \eqref{eq:5} .\\
 Let further
\begin{displaymath}
  \Dg=\{ \nu\in[U(-1),U(1)] \,| \,U^{\prime\prime}(x_\nu)=0\,\} \,.
\end{displaymath}
For $\nu\in\Dg$ we then define the operator 
$\Kg_{\nu}^{\mathcal D}$ as the Dirichlet realization in $L^2(-1,+1)$ of the differential operator 
\begin{displaymath}
  \Kg_{\nu}=-\frac{d^2}{dx^2}+\frac{U^{\prime\prime}}{U-\nu} \,.
\end{displaymath}
Hence, under the assumption (\ref{eq:5}),  we have
\begin{displaymath}
D(\Kg_{\nu}^\D)= H^2(-1,1)\cap H^1_0(-1,1)
\end{displaymath}
We can now state the main result. 
\begin{theorem}
  \label{thm:orr} 
  Let $U\in C^4([-1,1])$ satisfy \eqref{eq:5} and 
\begin{equation}
\label{eq:8}  
\inf_{\nu\in\Dg}\min  \sigma(\Kg_\nu^\D)>0 \,.
\end{equation} 
Then, there exist positive $C$, $\Upsilon$, and $\beta_0>1$ such
that for all $\beta>\beta_0$ and $\alpha,\lambda$ such that $ 0\leq \alpha$ and  $ \Re \hat{\Lambda}
  <\Upsilon\beta^{-1/3}+\alpha^2\beta^{-1}/2 $ (or equivalently for $\Re\lambda<\Upsilon\beta^{-1/3}-\alpha^2\beta^{-1}/2$),
$\B_{\lambda,\alpha,\beta}^\D$ is invertible and
   \begin{equation}
 \label{eq:9}
  \big\|(\B_{\lambda,\alpha,\beta}^\D)^{-1}\big\|+
        \Big\|\frac{d}{dx}\, (\B_{\lambda,\alpha,\beta}^\D)^{-1}\Big\|\leq
          C\,  \beta^{-5/6}  \,.
   \end{equation}
\end{theorem}
It should be mentioned that similar results are obtained in
\cite[Proposition 5.1]{chenetal23} for $1\leq\alpha$.  More precisely,
consider the Dirichlet realization $\Rg_\alpha^D$ of the operator (see
\cite{chenetal23,weetal18})
\begin{equation}
\label{eq:10}
\Rg_\alpha^D=\Big(-\frac{d^2}{dx^2}+\alpha^2\Big)^{-1}\Big[U\Big(-\frac{d^2}{dx^2}+\alpha^2\Big)
+U^{\prime\prime}\Big] \,, 
\end{equation}
where $(-d^2/dx^2+\alpha^2)^{-1}$ denotes the inverse of the Dirichlet
realization of \break $(-d^2/dx^2+\alpha^2)$ in $(-1,+1)$.\\
 Obviously, $ \Rg_\alpha^D$
is a bounded operator on  $H^2(-1,1)\cap H^1_0(-1,1)$.  

For a monotone shear flow, under the condition that $ \Rg_\alpha^D$ does
not have any embedded eigenvalues in the essential spectrum (Note that
$\sigma_{ess}(\Rg_\alpha^D)=[U(-1),U(1)]$ \cite{weetal18,st95}) or isolated
eigenvalues, it is proved in \cite{chenetal23} that $ \B_{\lambda,\alpha,\beta}^\D$
is invertible for $\Re\hat{\Lambda} <\Upsilon\beta^{-1/3}$ (with weaker bounds for the
inverse than in \eqref{eq:9}) and then some semigroup estimates are
deduced.

We note in addition that using \eqref{eq:9} we may proceed as in \cite[Section
9]{AH1} to obtain semigroup estimates as
well.  \\
Furthermore, for $U$ satisfying \eqref{eq:5}, the requirement
\eqref{eq:8} guarantees that  $\Rg_\alpha^\D$ does not  possess
any eigenvalues (embedded or ordinary).

It should also be noted that the definition of the Rayleigh operator below in
\eqref{eq:11} is different from the definition in
\cite{chenetal23}. More precisely we have
\begin{displaymath}
 \Rg_\alpha^D=\Big(-\frac{d^2}{dx^2}+\alpha^2\Big)^{-1}\A_{0,\alpha}^\D \mbox{ on } H^2 \cap H_0^1 \,. 
\end{displaymath}
Hence, if $(\phi,i\lambda)$ is an eigenpair of $\Rg_\alpha^D$ then $\phi\in{\rm
  ker}\,\A_{\lambda,\alpha}\cap H^2(-1,1)$. 

 In recent years there has been significant progress in the study of
the Orr-Sommerfeld operator \cite{Orr1907}, see
\cite{chen2020transition,AH1,AH2,jia2023uniform,chen2024enhanced} to
name just a few of the works addressing the linear operator only. A
significant body of literature deals with weakly non-linear analysis
of the laminar flow, see for instance
\cite{BGM,Mas17,chen2024transition}. Another recent work of interest
is \cite{jezequel2025orr} where the authors show (for $\alpha\gtrsim1$) that existence of
eigenvalues for the Orr-Sommerfeld operator depend on their existence
as eigenvalues of  the corresponding Rayleigh operator except for
the cases $\lambda=iU(\pm1)$. 

It should be clear for the reader at this stage that analysis of the
Rayleigh operator is of great interest when attempting to locate the
spectrum of the Orr-Sommerfeld operator in the large $\beta$ limit. We
mention here a few relevant works. In \cite{li03,li05} it is shown
that an eigenvalue of the Rayleigh operator embedded in the continuous
spectrum can exist only in the set $\Dg$ defined above. In
\cite{hietal14} it is demonstrated for holomorphic monotone flows and
some additional conditions, which we list here towards the end of the
next section, that eigenvalues for the Rayleigh operators can exist
only if \eqref{eq:8} is satisfied.

The rest of the contribution is arranged as follows: In the next
section we consider the Rayleigh operator and obtain for it some
inverse estimates away from its continuous spectrum. We also extend
the results in \cite{hietal14} to any $U\in C^3$ satisfying \eqref{eq:5}.
In Section \ref{sec:3} we repeat some results obtained for the resolvent of the
Schr\"odinger operator \eqref{eq:p29} in \cite{AH1,AH2} and obtain
some new estimates for it. In Section \ref{sec:4} we consider the Orr-Sommerfeld
operator and obtain for it inverse estimate in various subsets of the
parameter space. Finally, in the last section we complete the proof of
Theorem \ref{thm:orr}.

\section{Rayleigh estimates}
\label{sec:2}
We consider
$ U\in C^3([-1,1])$ satisfying $U^\prime\neq0$ on $[-1,1]$. Assume without any loss
of generality that $U^\prime>0$. We recall that, for $\nu\in[U(-1),U(1)]$, $x_\nu$ is the unique solution of 
 $U(x_\nu)=\nu$. 
For $\lambda=\mu+i\nu$, the Rayleigh operator (see 
\cite{AH1}) is the realization of
\begin{equation}
    \label{eq:11}
\A_{\lambda,\alpha}^D=(U+i\lambda)\Big(-\frac{d^2}{dx^2}+\alpha^2\Big) +U^{\prime\prime} \,,
\end{equation}
whose domain  is defined, for $\mu\neq0$ or   $\mu=0$, $\nu \not\in [U(-1),U(1)]$  on $H^2(-1,1)\cap H^1_0(-1,1)$
and  when $\mu=0$, $\nu \in [U(-1),U(1)]$ on the  set
\begin{equation}
\label{eq:12}
  D(\A_{i\nu,\alpha}^\D)= H^2((-1,1);|U-\nu|^2dx)\cap H^1_0(-1,1)\,.
\end{equation}
\begin{remark}
\label{rem:fredholm}
  It is proved in \cite[Proposition 4.13]{AH1} that that Fredholm
index $\A_{\lambda,\alpha}^D$ for $|\mu|>0$ is zero. Thus, proving its injectivity
would imply its invertibility as well.
\end{remark}

 We now discuss the invertibility of  $\A_{\lambda,\alpha}^\D$ and obtain estimates for $(\A_{\lambda,\alpha}^\D) ^{-1}$.
\begin{proposition}
  \label{prop:inviscid-boundedness-1} For any $p >1$ there exist
  positive $\mu_0$ and $C$, such that for any $\lambda=\mu + i \nu \in\C$ for which
  $\nu\in[U(-1),U(1)]$, $U''(x_\nu)\neq 0$, $0<|\mu|\leq \mu_0\, |U''(x_\nu)|$, and $\alpha\geq0$,
  $\A_{\lambda,\alpha}^\D$ is invertible and moreover, satisfy, for all $v\in
  W^{1,p}(-1,1)$,
\begin{subequations}
\label{eq:13}
  \begin{equation}
\frac{|U^{\prime\prime}(x_\nu)|}{\log\, (|\mu|^{-1} )}  \|(\A_{\lambda,\alpha}^\D)^{-1}v\|_{1,2} \leq C \, \|v\|_\infty \,,
\end{equation}
\begin{equation}
|U^{\prime\prime}(x_\nu)|\,\|(\A_{\lambda,\alpha}^\D)^{-1}v\|_{1,2} \leq C \, (\|v^\prime\|_p+\|v\|_\infty)\,,
\end{equation}
and 
\begin{equation}
|U^{\prime\prime}(x_\nu)|\,|\mu|^{1/p}\|(\A_{\lambda,\alpha}^\D)^{-1}v\|_{1,2} \leq C\, \|v\|_p\,.
\end{equation}
\end{subequations}
Similar estimates hold for $\nu\in\R\setminus[U(-1),U(1)]$ in which case we set
$|U^{\prime\prime}(x_\nu)|=1$ in \eqref{eq:13}. 
\end{proposition}
\begin{proof}
The proof is very similar to the proof of \cite[Proposition
4.14]{AH1}.  Since \eqref{eq:13} is trivial for $\nu \in \Dg$, we can assume throughout the proof that
$U''(x_\nu)\neq 0$. \\
  {\em Step 1:} {\it For  $1<p$ and $\mu \neq 0$ define $N_{m,p}^\pm$ by
\begin{displaymath}
  v \mapsto N_{m,p}^\pm(v,\lambda) := \min \Big(\Big\|(1\pm\cdot)^{1/2}\frac{v}{U+i\lambda}\Big\|_1,\|v\|_{1,p}\Big)\,.
\end{displaymath}
We prove that for all $p>1$ there exists $C>0$ such that,
for all $\varepsilon>0$, $\nu\in[U(-1),U(1)]$, and $0<|\mu|\leq 1$ it holds that 
 \begin{equation}
\label{eq:14} 
   |\phi(x_\nu)| \leq C\Big(\frac{\varepsilon^{-1/2}}{|U^{\prime\prime}(x_\nu)|} N_{m,p}^\pm (v,\lambda) +
   \Big(\Big|\frac{\mu}{U^{\prime\prime}(x_\nu)}\Big|^{1/2} +\varepsilon^{1/2}\Big)\|\phi^\prime\|_2\Big)\,,
 \end{equation}
 for all pairs $(\phi,v)\in D(\A_{\lambda,\alpha}^\D)\times W^{1,p}(-1,1)$ satisfying
 $\mathcal A_{\lambda,\alpha} \phi =v$.\\ \vspace{2ex}} 

We follow the same arguments of step 1 in the proof of
\cite[Proposition 4.14]{AH1}, paying special attention to the
dependence of the various constants on $U''(x_\nu)$.  An integration by
parts yields
\begin{displaymath}
 \Im\Big\langle\phi,\frac{v}{U-\nu+i\mu}\Big\rangle=-\mu\, \Big\langle\frac{U^{\prime\prime}}{(U-\nu)^2+\mu^2}\phi
 ,\phi\Big\rangle\,,
\end{displaymath}
which we can rewrite in the form
\begin{equation}
\label{eq:15}
  \Im\Big\langle\phi,\frac{v}{U-\nu+i\mu}\Big\rangle=-\mu\,
  \Big\langle\frac{U^{\prime\prime}(x_\nu)}{(U-\nu)^2+\mu^2}\phi ,\phi\Big\rangle-\mu\,
  \Big\langle\frac{[U^{\prime\prime}(x)-U^{\prime\prime}(x_\nu)]}{(U-\nu)^2+\mu^2}\phi ,\phi\Big\rangle \,.
\end{equation}
To estimate the last term we use an integration by parts to obtain 
\begin{displaymath}
  \Big\langle\frac{[U^{\prime\prime}(x)-U^{\prime\prime}(x_\nu)]}{(U-\nu)^2+\mu^2}\phi ,\phi\Big\rangle=
  \Big\langle\Big(\frac{U^{\prime\prime}(x)-U^{\prime\prime}(x_\nu)}{U^\prime(U-\nu)}|\phi|^2\Big)^\prime,\log[(U-\nu)^2+\mu^2]\Big\rangle\,,
\end{displaymath}
from which we conclude, using \eqref{eq:5}, the fact that
$\|\log[(U-\nu)^2+\mu^2]\|_q$ is uniformly bounded for
$(\mu,\nu)\in[-1,1]\times[U(-1),U(1)]$ for $q>1$, and Poincar\'e's inequality that
\begin{equation}
  \label{eq:16}
 \Big|\Big\langle\frac{[U^{\prime\prime}(x)-U^{\prime\prime}(x_\nu)]}{(U-\nu)^2+\mu^2}\phi ,\phi\Big\rangle\Big|\leq
 C\|\phi\|_\infty\|\phi^\prime\|_2
\end{equation}
As 
\begin{displaymath}
  |\phi (x)|^2\geq \frac{1}{2}|\phi(x_\nu)|^2 -
  |\phi(x)-\phi(x_\nu)|^2 \,,
\end{displaymath}
we may use \eqref{eq:15} and \eqref{eq:16} to obtain
\begin{multline} 
\label{eq:17}
 |  \Im\Big\langle\phi,\frac{v}{U-\nu+i\mu}\Big\rangle| \geq \\|\mu U^{\prime\prime}(x_\nu)| \,
   \Big\langle\frac{1}{(U-\nu)^2+\mu^2},\frac 12 |\phi(x_\nu)|^2 - |\phi(x)-\phi(x_\nu)|^2
   \Big\rangle\\ -C|\mu|\,\|\phi\|_\infty\|\phi^\prime\|_2\,.
\end{multline}
We note that, for any $1<p$, there exists $C>0$ such that
\begin{displaymath}
\begin{array}{ll}
    \Big|\Big\langle\phi,\frac{v}{U+i\lambda}\Big\rangle\Big|&  =
    \Big|\Big\langle\Big(\frac{\phi\bar{v}}{U^\prime}\Big)^\prime,\log\,(U+i\lambda)\Big\rangle\Big| \\
   & \leq C\, \left(\|\phi^\prime\|_2\|v\|_\infty+\|\phi\|_\infty\|v^\prime\|_p \right)\\ & \leq C\, \|\phi^\prime\|_2\|v\|_{1,p}\,.
   \end{array}
  \end{displaymath}
  Here, as above, we use the boundedness of the $L^q$ norm of
  $\log[(U-\nu)^2+\mu^2]$ for $q=p/(p-1)$.  On the other hand, since
\begin{displaymath}
|\phi(x)|=|\phi (x)-\phi(\pm 1)| \leq \| \phi^\prime\|_2 (1\pm x)^\frac
    12\,, 
\end{displaymath} 
we may conclude that
\begin{displaymath}
      \Big|\Big\langle\phi,\frac{v}{U+i\lambda}\Big\rangle\Big| \leq 
      \|\phi^\prime\|_2 \,
      \Big\|(1\pm\cdot)^{1/2}\, \frac{v}{U+i\lambda}\Big\|_1 
\end{displaymath}
and hence, there exists $C>0$, such  that
\begin{equation}
  \label{eq:18}  
\Big|\Big\langle\phi,\frac{v}{U+i\lambda}\Big\rangle\Big| \leq  C\, \|\phi^\prime\|_2 \, N_{m,p}^\pm(v,\lambda) \,.
\end{equation}
Substituting the above into \eqref{eq:17} yields
\begin{equation*}  
\begin{array}{ll}
  \frac{|\mu U^{\prime\prime}(x_\nu)| }{2} |\phi(x_\nu)|^2 \Big\|\frac{1}{U+i\lambda}\Big\|_2^2 &\leq
  |\mu U^{\prime\prime}(x_\nu)|  \Big\|\frac{\phi-\phi(x_\nu)}{U+i\lambda}\Big\|_2^2 \\ &\qquad  +C|\mu|\ \|\phi\|_\infty\|\phi^\prime\|_2
  +  C\|\phi^\prime\|_2 N_{m,p}^{\pm} (v,\lambda)\,.
  \end{array}
\end{equation*}
We now observe, as in the proof of
 \cite[Eq. (4.60)]{AH1}, that for some $\hat C >0$ 
\begin{displaymath}
  \Big\|\frac{1}{U+i\lambda}\Big\|_2^2 \geq   \frac {1}{\hat C\,|\mu|} \,.
\end{displaymath}
 Hence, for another constant $C>0$, we get
\begin{multline}\label{eq:19}
   |\phi(x_\nu)|^2 \leq C\Big[ |\mu|
   \Big\|\frac{\phi-\phi(x_\nu)}{U+i\lambda}\Big\|_2^2\\
   +\Big|\frac{\mu}{U^{\prime\prime}(x_\nu)}\Big|\,\|\phi\|_\infty\|\phi^\prime\|_2 +\frac{1}{|U^{\prime\prime}(x_\nu)|} \|\phi^\prime\|_2N_{m,p}^\pm(v,\lambda)\Big]\,.
\end{multline}
To estimate the first  term on the right-hand-side of
\eqref{eq:19} we use Hardy's inequality, as in the proof of
\cite[Eq. (4.57)]{AH1} (see the end of the proof after (4.60)),  to obtain
\begin{equation}
\label{eq:20}
\Big\|\frac{\phi-\phi(x_\nu)}{U+i\lambda}\Big\|_2^2 \leq C \,   \| \phi^\prime\|_2^2\,,
\end{equation}
which when substituted into \eqref{eq:19} readily yields \eqref{eq:14}
via Cauchy's inequality and Sobolev embeddings.  \vspace{1ex}

{\em Step 2:} {\it For $\nu \in (U(-1),U(+1))$, let
  $d_\nu=\min(1-x_\nu,1+x_\nu)$. We prove that for any $A >0$,  $p>1$
  and $\hat d >0$, there exist $C$ and $\mu_0$ such that, for $\alpha^2 \leq A$,
  $ |\mu|\leq \mu_0\, |U^{\prime\prime}(x_\nu)|$, and $d_\nu \geq \hat d$\,,
\begin{equation}
  \label{eq:21}
\|\phi\|_{1,2}\leq  \frac{C}{|U^{\prime\prime}(x_\nu)|}\, N_{m,p}^\pm(v,\lambda)\,.
\end{equation}
holds for any pair $(\phi,v)$ in $D(\A_{\lambda,\alpha}) \times W^{1,p}(-1,1)$ satisfying
$\mathcal A_{\lambda,\alpha}\, \phi=v$. }
\vspace{3ex}
\\
As in step 1 we follow now step 2 in the proof of \cite[Proposition
4.14]{AH1}, paying special attention to the dependence of
constants on $U''(x_\nu)$.\\
\noindent Let
$\chi\in C_0^\infty(\R,[0,1])$ satisfy 
\begin{displaymath}
  \chi(x)=
  \begin{cases}
    1 & |x|<1/2 \\
    0 & |x|>3/4 \,.
  \end{cases}
\end{displaymath}
Let $\chi_d(x) = \chi((x-x_\nu)/d)$ (with $d=d_\nu$) and set
\begin{equation}
\label{eq:22}
  \phi=\varphi + \phi(x_\nu)\chi_d \,.
\end{equation}
Note that by the choice of $d$, $\varphi$ satisfies also the boundary condition at $\pm 1\,$.\\
It can be easily verified that
\begin{displaymath}
  \A_{\lambda,\alpha}\varphi =v + \phi(x_\nu)\big((U+i\lambda)(\chi_d^{\prime\prime}-\alpha^2\chi_d) -U^{\prime\prime}\chi_d\big)
  \,.
\end{displaymath}
By construction,  $w:=(U-\nu)^{-1}\varphi$ belongs to $H^2(-1,+1)$. As in \cite[Eq.
(4.63)]{AH1} we show that
\begin{multline} 
\label{eq:23}
  \|(U-\nu)w^\prime\|_2^2 + \alpha^2\|\varphi\|_2^2 =\Big\langle \varphi, \frac{v}{U+i\lambda}\Big\rangle -\langle w,\phi(x_\nu)U^{\prime\prime}\chi_d\rangle\\
  +  \phi(x_\nu)\langle\varphi,\chi_d^{\prime\prime}-\alpha^2\chi_d\rangle  +i\mu\Big\langle w,\frac{U^{\prime\prime}\phi}{U+i\lambda}\Big\rangle\,.
\end{multline}
By \eqref{eq:20} (with $\phi$ replaced by $\varphi$) we have that
\begin{equation}
\label{eq:24}
\Big|\Big\langle \varphi, \frac{v}{U+i\lambda}\Big\rangle\Big|\leq  C \|\varphi^\prime\|_2\, \|v\|_2\leq
\hat C\, (\|\phi^\prime\|_2+d^{-1/2}|\phi(x_\nu)|)\|v\|_2 \,.
\end{equation}
We continue in the same manner as in the proof following
\cite[Eq. (4.69)]{AH1} to obtain for any $\varepsilon_1\in(0,1)$, 
\begin{equation}
  \label{eq:25} 
  \|(U-\nu)w^\prime\|_2^2 + \alpha^2\|\varphi\|_2^2 \leq  C \, \Big((\varepsilon_1+|\mu|^{1/2})\|\phi^\prime\|_2^2+ 
  \frac{1}{\varepsilon_1d}|\phi(x_\nu)|^2 +\varepsilon_1^{-1}N_{m,p}^\pm(v,\lambda)^2\Big)\,.
\end{equation}
By Hardy's inequality, Poincar\'e's inequality, and
\eqref{eq:14}  we obtain, for \break $0 < |\mu| \leq 1$,  $\varepsilon\in (0,1)$, and $\varepsilon_1\in (0,1)$, 
that
\begin{displaymath}
  \|w\|_2 \leq C \Big( \Big[ |\mu|^{1/4}+\varepsilon_1^{1/2}+\frac{\varepsilon^{1/2}+\Big|\frac{\mu}{U^{\prime\prime}(x_\nu)}\Big|^{1/2}}{[\varepsilon_1d]^{1/2}}\Big] \|\phi^\prime\|_{2}
 + [|U^{\prime\prime}(x_\nu)|^2 \varepsilon_1d\varepsilon]^{-1/2}N_{m,p}^\pm(v,\lambda)\Big)\,.
\end{displaymath}
Selecting $\varepsilon=\varepsilon_1^2$ and continuing as in   \cite{AH1} yields the existence of
$\mu_0>0$ such that for any $|\mu/U^{\prime\prime}(x_\nu)|\leq \mu_0$,  $\varepsilon_1\in (0,1)$, and  $d\geq \hat d$ 
\begin{displaymath}
    \|\phi^\prime\|_2\leq C(\hat d) \Big(\frac{\varepsilon_1^{-3/2}}{|U^{\prime\prime}(x_\nu)|} N_{m,p}^\pm(v,\lambda)+  (|\mu|^\frac 14 \,+  |\mu|^{\frac 12} \epsilon_1^{-\frac 12}  + \varepsilon_1^{1/2})\|\phi^\prime\|_2\Big)\,.
\end{displaymath}
Hence, we can choose first $\epsilon_1$ and  then a possibly smaller  $\mu_0 >0$ such that 
\eqref{eq:21} follows for  $|\mu|\leq \mu_0 \,  |U^{\prime\prime}(x_\nu)|$ and $d\geq \hat d$. \\

The proof of the next steps 3-5 is entirely identical with the proof
of \cite[Proposition 4.14- steps 3-5]{AH1} and is therefore skipped.
Note indeed that these steps do not rely on the assumption $U^{\prime\prime}\neq0$
made in \cite{AH1}.  Since these steps are valid for $|\mu| \leq 1$ they are
also valid for $|\mu| \leq \mu_0|U^{\prime\prime}(x_\nu)| $ by choosing
  $\mu_0<\|U^{\prime\prime}\|_\infty^{-1}$.

{\em Step 3:} There exists $\alpha_0>0$ such that 
for any $\alpha>\alpha_0$ and $|\mu|\leq 1$
\begin{displaymath}
\|\phi\|_{1,2}\leq C N_{m,p}^\pm(v,\lambda)\,.
\end{displaymath}
\vspace{2ex}

{\em Step 4:} We prove that there exist
  $\hat d_0>0$, $\mu_0>0$  and
$C >0$ such that, for all $d \leq \hat d_0$, $\nu\in[U(-1),U(1)]$, $\alpha \geq 0$  and $|\mu| \leq \mu_0$, 
\begin{equation}
\label{eq:26} 
\|\phi\|_{1,2} \leq  C\, N_{m,p}^\pm(v,\lambda)\,.
\end{equation}
holds for any pair $(\phi,v)$ such that $\mathcal A_{\lambda,\alpha} \phi=v$.
\vspace{2ex}

{\em Step 5:} We prove that there exist $C>0$ and $\mu_0>0$ such that \eqref{eq:26}
holds for all  $\nu\in\R\setminus[U(-1),U(1)]$ and $|\mu|\leq \mu_0$. 
\vspace{2ex}

{\em Step 6:} {\it Prove \eqref{eq:13}. }

In comparison with Step 6 in \cite{AH1}, we have established so far that
there exist $C$ and $\mu_0$ such that if $|\mu| \leq \mu_0 |U(x_\nu)|$ then
\begin{equation} 
\label{eq:27}
|U^{\prime\prime} (x_\nu)| \,\|\phi \|_{1,2} \leq C \, N^{\pm}_{m,p} (v,\mu + i \nu)\,.
\end{equation}
This proves the injectivity of $\A_{\lambda,\alpha}$ for $U''(x_\nu)\neq 0$ and 
  by Remark \ref{rem:fredholm} its invertibility as well.
Consequently, all we need to do is to estimate $N^{\pm}_{m,p} (v,\mu + i
\nu)$ for the derivation of \eqref{eq:13}. This can be done in precisely
the same manner as in the proof of Step 6 in \cite{AH1}.
\end{proof}
We note that in \cite[Theorem 2.3]{li05} it is proved that for $\nu_0 \in
[U(-1),U(1)]$ and ($\{\phi_n, \lambda_n,\alpha_n\}_{n\in \N}\in(H^2 (-1,1)\cap
H_0^1(-1,1))(-1,+1)\times \mathbb C\times\R_+$ such that \break $-i\lambda_n \not\in
[U(-1),U(1)]$, $\lim_{n\to +\infty}\lambda_n =i \nu_0$, $\phi_n \in \text{Ker}\,
\A_{\lambda_n,\alpha_n}^D\setminus \{0\}$ , it holds that $\nu_0\in \Dg$ (or equivalently $U'' (x_{\nu_0})
=0$). Proposition \ref{prop:inviscid-boundedness-1} gives the same
result together with an inverse estimate.

 We continue by obtaining inverse estimate for $\A_{\lambda,\alpha}$ for $\nu$
  in the vicinity of $\Dg$.
\begin{lemma}
\label{lem:positivity}
 Under Assumption \ref{eq:5}, suppose that for some $\nu_0\in \Dg$,
 \begin{displaymath}
\sigma_{\nu_0}:=\inf  \sigma(\Kg_{\nu_0}^\D) >0\,.
 \end{displaymath}
Then, for any $p>1$ there
  exists a positive $\delta$ such that for any $\lambda=\mu + i \nu$ with $0<|\mu|<\delta$
  and $|\nu-\nu_0|<\delta$, $\A_{\lambda,\alpha}^\D$ is invertible and  there exists 
  $C>0$ such that, for all $v\in W^{1,p}(-1,1)$, 
  \begin{equation}
 \label{eq:28}
\| (\A_{\lambda,\alpha}^\D)^{-1} v\|_{1,2} \leq C \, N_{m,p}^\pm(v,\lambda) \,.
  \end{equation}
\end{lemma}
\begin{proof}
  Let $\phi\in D(\A_{\lambda,\alpha}^\D)$ satisfy $\A_{\lambda,\alpha}\phi=v$. 
  Note that by assumption $(U-\nu_0)^{-1}U^{\prime\prime}\in C^1([-1,1])$ and let 
\begin{displaymath}
  s:=\min_{x\in[-1,1]}  \frac{U^{\prime\prime}(x)}{U(x)-\nu_0} \,.
\end{displaymath}
Then, we have 
\begin{equation}
\label{eq:29}
  \langle \phi,(\Kg_{\nu_0}+\alpha^2)\phi\rangle\geq\frac{\sigma_{\nu_0} }{ 2\sigma_{\nu_0} +|s|-s}\|\phi^\prime\|_2^2+
  \Big[\frac{\sigma_{\nu_0} }{2}+\alpha^2\Big]\|\phi\|_2^2\,.
\end{equation}

On the other hand  we can write 
  \begin{equation}
\label{eq:30}
    (\Kg_{\nu_0}+\alpha^2)\phi=\frac{v}{U+i\lambda}+U^{\prime\prime}\Big[\frac{1}{U-\nu_0}-\frac{1}{U+i\lambda}\Big] \, \phi\,.
  \end{equation}
Hence, using \eqref{eq:29} and Poincar\'e's inequality, there exists $C>0$ such that
\begin{equation}
\label{eq:31}
  \|\phi\|_{1,2}^2\leq C\Big[\Re\Big\langle\phi,\frac{v}{U+i\lambda}\Big\rangle+ \Re\Big\langle\phi,
 U^{\prime\prime}\phi\Big(\frac{1}{U-\nu_0}-\frac{1}{U+i\lambda}\Big)\Big\rangle\Big]\,.
\end{equation}
For the second term on the right-hand-side it holds that
\begin{multline}
\label{eq:32}
  \Big|\Big\langle\phi,U^{\prime\prime}\phi\Big(\frac{1}{U-\nu_0}-\frac{1}{U+i\lambda}\Big)\Big\rangle\Big|=
  \Big|\Big\langle\Big(\frac{|\phi|^2U^{\prime\prime}}{U^\prime}\Big)^\prime,\log\Big(\frac{U+i\lambda}{U-\nu_0}\Big)\Big\rangle\Big| \\
  \leq C\, \|\phi\|_\infty\|\,\|\phi^\prime\|_2\Big\|\log\Big(\frac{U+i\lambda}{U-\nu_0}\Big)\Big\|_2 \,.
\end{multline}
A simple computation yields
\begin{displaymath}
  \Big\|\log\Big(\frac{U+i\lambda}{U-\nu_0}\Big)\Big\|_2\leq C|\lambda-i\nu_0|^{1/2}\,\Big|\log\Big(\frac{1}{|\lambda-i\nu_0|} \Big)\Big|
\end{displaymath}
Substituting the above inequality, together with \eqref{eq:32} into
\eqref{eq:31} we obtain, with the aid of \eqref{eq:18} (which holds
when $U''(x_\nu)=0$ as well) and Poincar\'e's inequality,
\begin{displaymath}
  \|\phi\|_{1,2}^2\leq C\, |\lambda-i\nu_0|^{1/2}\,\Big|\log\Big(\frac{1}{|\lambda-i\nu_0|}\Big)\Big|\|\phi^\prime\|_2^2+\|\phi^\prime\|_2 \, N_{m,p}^\pm(v,\lambda) \,.
\end{displaymath}
For sufficiently small $\delta$ we readily obtain \eqref{eq:28}.  Using
Remark \ref{rem:fredholm} once again we can then conclude that
$\A_{\lambda,\alpha}$ is invertible.
invertible. 
\end{proof}

Proposition \ref{prop:inviscid-boundedness-1} and Lemma
\ref{lem:positivity} readily yield the following conclusion.
\begin{proposition}
\label{cor:strip-bounded-inverse}
Suppose that  $\inf_{\nu\in\Dg}\min \sigma(\Kg_\nu)=\sigma_0>0$. Then, for any $p>1$ there exist $\mu_0>0$ and $C>0$ such that
  for all $0<|\mu|\leq \mu_0$, $\nu\in\R$, and $\alpha\in\R_+$ \eqref{eq:28} holds true.
\end{proposition}
\begin{proof}
Note that in the set
\begin{displaymath}
  \D_\delta=\{\nu\in\R \,| \,d(\nu,\Dg)\geq\delta\,\} \,,
\end{displaymath}
it holds that $|U^{\prime\prime}(x_\nu)|\geq m_\delta>0$. Hence we can bound
$(\A_{\lambda,\alpha}^\D)^{-1}$ for $\nu\in\Dg_\delta$ using Proposition
\ref{prop:inviscid-boundedness-1}, whereas for $\nu\in[U(-1),U(1)]\setminus\Dg_\delta$
we can apply Lemma \ref{lem:positivity}. 
\end{proof}
From Proposition \ref{cor:strip-bounded-inverse} we may conclude that
\begin{corollary}
\label{cor:no-eigenvalues} 
  Under the conditions of Proposition \ref{cor:strip-bounded-inverse}
  there exists $\mu_0>0$ such  that for $\lambda =\mu+ i \nu$, $|\mu| \leq \mu_0$,  and
  $\alpha \geq 0$,  if $\phi\in H^2(-1,1)\cap H^1_0(-1,1)$ satisfies 
  $\A_{\lambda,\alpha}\phi=0$  and then $\phi\equiv0$. 
\end{corollary}
\begin{proof}~\\

  The case $0<|\mu|\leq \mu_0$ follows immediately from Proposition
  \ref{cor:strip-bounded-inverse}.\\  Consider now the case $\mu=0$.
  Suppose that for some $\phi\in H^2(-1,1)\cap H^1_0(-1,1)$, $\nu\in\R$, and
  $\alpha\in\R_+$ it holds that $\A_{i\nu,\alpha}\phi=0$. Then, we may write, for some
  $0<\mu<\mu_0$,
  \begin{displaymath}
    \A_{\lambda,\alpha}\phi=-i\mu\, (\phi^{\prime\prime}-\alpha^2\phi) \,.
  \end{displaymath}
By \eqref{eq:28} we then have
\begin{displaymath}
  \|\phi\|_{1,2}\leq C\mu\, \Big\|\frac{\phi^{\prime\prime}-\alpha^2\phi}{(U-\nu)+i\mu}\Big\|_1 \leq
  C\, |\mu|^{1/2}\,\|\phi^{\prime\prime}-\alpha^2\phi\|_2 \,.
\end{displaymath}
Letting $\mu\to0$ yields $\phi\equiv0$. 
\end{proof}

The following rather standard observations  are also necessary in the
sequel 
\begin{lemma}
\label{lem:large-eigenvlaue}\strut
\begin{itemize}
\item 
There exist $\lambda_0>0$ and $C>0$ such that for any $|\lambda|\geq \lambda_0$,
$\A_{\lambda,\alpha}^\D$ is invertible and for any $v\in L^2(-1,1)$
  \begin{equation}
    \label{eq:33}
\|\big(A_{\lambda,\alpha}^D\big)^{-1}v\|_{1,2} \leq \frac{C}{|\lambda|}\|v\|_2 \,.
  \end{equation}
\item For any $p>1$, there exists $\alpha_0>0$ and $C>0$ such that for any
  $\lambda\in\C$, $\alpha \geq \alpha_0$, $\A_{\lambda,\alpha}^\D$ is invertible and \eqref{eq:28}
  holds true.
 \end{itemize}
\end{lemma}
\begin{proof}
  The proof of \eqref{eq:33} follows immediately from the equation
\begin{displaymath}
  -\phi^{\prime\prime}+\alpha^2\phi + \frac{U^{\prime\prime}}{U+i\lambda}\phi=\frac{v}{U+i\lambda}\,. 
\end{displaymath}

From step 3 of the proof of Proposition
\ref{prop:inviscid-boundedness-1} we conclude that there exists
$\alpha_0>0$ such that \eqref{eq:28} holds true for all $\alpha \geq \alpha_0$ and
$|\mu|<1$. For $|\mu|\geq1$ we may conclude as in the proof of \eqref{eq:33}
that \eqref{eq:28} holds whenever $\alpha^2\geq 2 \|U^{\prime\prime}\|_\infty$.
\end{proof}

We can now prove the main result of this section. A similar result has
been proved in \cite{hietal14}, assuming in addition that $U$ is real analytic  on
$[-1,+1]$ and that $U^{(3)}(x_\nu)\neq 0$ for any $\nu \in \Dg$. (See their
assumptions (A1) and (A2) just above Theorem 3.1 .)
\begin{theorem}
  \label{thm:invicid} 
Let $\Rg_\alpha^D$ be defined as in \eqref{eq:10}  for $U\in C^3[-1,1]$
satisfying \eqref{eq:5}.
If
\begin{equation}
\label{eq:34}
  \inf_{\nu\in\Dg}\min  \sigma(\Kg_\nu^\D)>0 \,,
\end{equation}
then
\begin{displaymath}
  \sigma(\Rg_\alpha^D)= [U(-1),U(1)] 
\end{displaymath}
for all $\alpha\geq0$. Furthermore, there are no embedded eigenvalues of
$\Rg_\alpha^D$ in $[U(-1),U(1)]$.
\end{theorem}
\begin{proof}
Let $\Omega =\{(\mu,\nu,\alpha) \mbox{ s.t.  } \mu \geq 0 \mbox{ and }
  (\nu -i\mu)\in \sigma(\Rg_\alpha^D)$. By the foregoing analysis (in particular Corollary \ref{cor:no-eigenvalues}), we know under
  the assumption of the theorem that 
\begin{displaymath}
\Omega \subset[\mu_0,\lambda_0] \times  [-\lambda_0,\lambda_0] \times[0,\alpha_0]  \,.
\end{displaymath} 
Suppose that $\Omega$ is not empty and then introduce
\begin{equation}
\label{eq:35}
A :=\sup \{\alpha, (\lambda,\alpha)\in \Omega\}\,.
\end{equation}
Notice that $A <+\infty$. By definition, there exists a sequence $(\alpha_k,\lambda_k)
\in \Omega$ and a sequence $\{\phi_k\}_{k=1}^\infty\subset H^2(-1,1)\cap H^1_0(-1,1)$ such that
$\|\phi_k\|_2=1$, $\A_{\lambda_k,\alpha_k}\phi_k=0$, and $ \lim_{k\to\infty} \alpha_k =A$. Using the
relative compactness of $\Omega$ in $ [0,\alpha_0] \times [\mu_0,\lambda_0] \times [-\lambda_0,\lambda_0]$ we
can assume, after extraction of a subsequence, that $\lambda_k$ is
convergent to some $\lambda_{\infty}$ and that $\phi_k$ is strongly convergent to
some $\phi_{\infty}$ in $H^1$ (hence in $H_0^1$) which satisfies
$\A_{\lambda_\infty,\alpha_{\infty} } \phi_\infty =0$ in $\mathcal D'(-1,+1)$.  Hence $\phi_\infty$ belongs
to $H_0^1\cap H^2$ and $(\lambda_\infty,A) \in \Omega$.

Finally, we show the contradiction with the definition of $A$.  Since (see \cite{st95,weetal18})
$$\sigma_{ess}(\Rg_\alpha^D)=[U(-1),U(1)]\,,$$  $i \lambda_\infty$ is an
isolated eigenvalue of finite multiplicity $m$ of $\Rg_A^\D$ at $\alpha=A$.
Since $\Rg_\alpha^D$ depends holomorphically on $\alpha$, by \cite[Theorem
1.7]{ka80} we may obtain all its eigenvalues, satisfying $i\lambda(A)=i\lambda_\infty$
through the spectral analysis of a $C^0$- family on $(A-\delta,A+ \delta)$ of
$m\times m$ matrices $M_\alpha$ (see also \cite{dusc58}). The eigenvalues appear
as the roots of the polynomial $\lambda \mapsto {\rm det} (M_\alpha - \lambda)$ whose
coefficients are continuous with respect to $\alpha$.  By continuity of the
roots, we find for $\alpha=A+\delta/2$ an eigenvalue of $\Rg_{A+\frac \delta 2}$
close to $\lambda_\infty$, contradicting, therefore, the definition of $A$ in
\eqref{eq:35}.
\end{proof}

\section{One-dimensional Schr\"odinger estimates}
\label{sec:3}
 In this section we obtain two improved resolvent estimates,
  with respect to \cite{AH1}, for the Schr\"odinger operator
  $-d^2/dx^2+i\beta U$.

\begin{remark}
\label{rem:AH2}
We shall use in the sequel results from \cite{AH2}. It should be noted
that the laminar velocity profile $U$ there is not strictly monotone,
as in the present contribution, and a single extremal point is assumed
at $x=0$.  Nevertheless, this extremal point has a significant effect
on the resolvent of the Schr\"odinger operator $-d^2/dx^2+i\beta U$ only
when $|U(0)-\nu|\ll1$. Hence, in the sequel, we use variants of the
estimates in \cite{AH2}.
\end{remark}

Let $\{\nu_k\}_{k=1}^\infty$ denote the zeroes of Airy's function
$\Ai:\R\to\R$. Let further $\LL_\beta^D:H^2(-1,1)\cap H^1_0(-1,1)\to L^2(-1,1)$ be
given by
\begin{equation}
\label{eq:36}
  \LL_\beta^D=-\frac{d^2}{dx^2}+i\beta U \,.
\end{equation}
We begin with the following
simple extension of \cite[Proposition 5.2]{AH1}
\begin{lemma}
\label{lem:standard}
 Let $U\in C^2[-1,1]$ satisfy \eqref{eq:5}.
  Let  further $\Upsilon<|\nu_1|/2$. 
 Then,   there exist positive $\beta_0$ and $C$ such that for all $\lambda =\mu +i
 \nu$ such that $\beta>\beta_0$ and $\mu\beta^{1/3}<\Upsilon$, $\LL_\beta^D-\beta\lambda$ is invertible and
 satisfies 
  \begin{equation}
    \label{eq:37}
\|(\LL_\beta^D-\beta\lambda)^{-1}\|+r_\beta^{1/2}\Big\|\frac{d}{dx}(\LL_\beta^D-\beta\lambda)^{-1}\Big\|\leq C\, r_\beta \,,
  \end{equation}
  with 
   \begin{displaymath}
   r_\beta=\min(\Upsilon|\mu\beta|^{-1},\beta^{-2/3})\,.
  \end{displaymath}
\end{lemma}
\begin{proof}
  We prove only the case $\mu<-\Upsilon\beta^{-1/3}$, otherwise we can use
  \cite[Proposition 5.2]{AH1}. Let $v\in H^2(-1,1)\cap
  H^1_0(-1,1)$ and $g\in L^2(-1,1)$ satisfy
  \begin{displaymath}
    (\LL_\beta^D-\beta\lambda)v=g\,.
  \end{displaymath}
An integration by parts yields
\begin{displaymath}
  \|v^\prime\|_2^2+|\mu|\beta\|v\|_2^2 =\Re\langle v,g\rangle\,,
\end{displaymath}
from which we easily conclude that
\begin{displaymath}
  \|v\|_2 \leq \frac{1}{|\mu|\beta}\|g\|_2\,.
\end{displaymath}
In addition, we may write
\begin{displaymath}
  \|v^\prime\|_2^2 \leq \|v\|_2\,\|g\|_2\leq\frac{1}{|\mu\beta|} \|g\|_2^2\,,
\end{displaymath}
which completes the proof of the lemma.
\end{proof}

For the convenience of the reader we adapt to the present context two
results from \cite{AH2} (see Remark \ref{rem:AH2}) which we frequently
use in the sequel.  We begin with \cite[Proposition 3.6]{AH2}
\begin{proposition}
  \label{Dirichlet-L1-H1-0} 
  Let $U\in C^2([0,1])$ satisfy \eqref{eq:5} and $\lambda_0>0$.
  Then there exist $\Upsilon>0$, $C>0$, and $\beta_0>0$ such that, for $\beta\geq
  \beta_0$, $|\nu|\leq \lambda_0$, $ \mu \leq\Upsilon\beta^{-1/3}$, and $f\in H^1(0,1)$ we have
\begin{equation}
\label{eq:38}
  \Big\| (\LL_\beta^D-\beta\lambda)^{-1} f+i\frac{f(x_\nu)}{\beta[U-\nu- i \max(-\mu,\beta^{-1/3})]} \Big\|_1
  \leq C\, \beta^{-1}\|f\|_{1,2}\,.
\end{equation}
\end{proposition}

Next, we bring the corresponding adapted version of   \cite[Proposition
3.8]{AH2}
\begin{proposition} 
\label{Prop:3.8AH2}
Let $U\in C^3([0,1])$ satisfy \eqref{eq:5}, $\lambda_0>0$ and
$\Upsilon<\nu_1$. Then there exist $C>0$, $\beta_0>0$ such that,
for all $\beta \geq \beta_0$,
      \begin{multline} \label{eq:39}
\sup_{
  \begin{subarray}{c}
   \mu  \leq\Upsilon\beta^{-1/3} \\
|\nu| \leq \lambda_0
  \end{subarray}} \Big(
\|(\LL_\beta^D-\beta\lambda)^{-1}(U-\nu) f\|_2  +
\beta^{-1/2}\Big\| \frac{d}{dx} (\LL_\beta^D-\beta\lambda)^{-1} (U-\nu) f \Big\|_2 \Big) \\ \leq   C \beta^{-1}\|f\|_2\,. 
  \end{multline}
\end{proposition}

Let further $\LL_\beta^\zeta$ be the differential
 operator $-d^2/dx^2 + i \beta U$ with domain
  \begin{equation}
\label{eq:40}
    D(\LL_\beta^\zeta)= \{ u\in H^2(-1,1)\,| \, \langle\zeta_\pm ,u\rangle=0\,\} \,.
  \end{equation}
 where $(\zeta_{-},\zeta_+)\in[H^1(-1,1)]^2$ are linearly independent but may
 depend on $\beta$. For convenience we require that $\zeta_\pm$ satisfy 
  \begin{equation}
\label{eq:41}   
    \begin{bmatrix}
      \zeta_+(1) & \zeta_-(1) \\
      \zeta_+(-1) & \zeta_-(-1)
    \end{bmatrix}
=
\begin{bmatrix}
  1 & 0 \\
  0 & 1
\end{bmatrix}
\,.
  \end{equation}

  Let (see  \cite[Eq. (6.8]{AH1} though the normalization there is
  different and \cite[Eq. (8.86)]{AH1})
    \begin{subequations}
\label{eq:42}
    \begin{equation} 
\hat{\psi}_-(x)= \frac{{\rm Ai}\big((J_- \beta)^{1/3}e^{
    i\pi/6}\big[(1+x)+iJ_-^{-1}\lambda_-\big]\big)}
{{\rm Ai}\big(J_-^{-2/3} \beta^{1/3}e^{ i2\pi/3}\lambda_-\big)}\Theta_-\,,
\end{equation}
and 
\begin{equation}
\overline{\hat{\psi}_+ (x)}=\frac{{\rm Ai}\big((J_+
  \beta)^{1/3}e^{ i\pi/6}\big[(1- x)+ iJ_+ ^{-1}\bar \lambda_+ ]\big)}
{{\rm Ai}\big(J_+^{-2/3} \beta^{1/3}e^{ i2\pi/3}\bar \lambda_+\big)}\Theta_+ \,.
\end{equation}
\end{subequations} 
    where
\begin{displaymath}
J_\pm = U^\prime(\pm 1)\,,
\end{displaymath}
\begin{equation}
\label{eq:43}
  \lambda_\pm = \mu - i(U(\pm 1)-\nu)\,,
\end{equation}
and
\begin{displaymath}
  \Theta_\pm (x)=\eta(1\mp x)\,.
\end{displaymath}
 The cutoff function
 $\eta\in C^\infty(\R_+,[0,1])$ satisfies
 \begin{displaymath}
   \eta(x)=
   \begin{cases}
    1 & x<\frac{1}{2} \\
    0 & x>1\,.
   \end{cases}
 \end{displaymath}
Let now $A_0$ denote the holomorphic extension to $\mathbb C$ of
\begin{equation}
\label{eq:44}
 x \mapsto  A_0(x)=e^{i\pi/6}\int_x^{+\infty}\Ai(e^{i\pi/6}t)\,dt\,,
\end{equation}
 and
\begin{displaymath}
  \Sg_\lambda =\{\,z\,|\,A_0 (i z)=0\,\}\,. 
\end{displaymath}
Finally we set  (see  \cite[Eq. (6.10)]{AH1})
\begin{equation}
\label{eq:45}
\vartheta_1^r:=  \inf_{z\in\Sg_\lambda} \Re z\,.
\end{equation}
\begin{proposition}
\label{lem:boundary}
Let $J_m=\min(J_+,J_-)$, $C_0>0$, and $\lambda_0>0$. Then, there exist 
$\Upsilon>0$, $\beta_0>0$, and $C>0$ such that, for all $\beta\geq \beta_0$ and $\lambda =\mu + i \nu \in
\C$ satisfying $|\lambda|\leq \lambda_0$ and
\begin{equation}
  \label{eq:46}
\beta^{1/3}\mu\leq J_m^{2/3}\Upsilon \,,
\end{equation}
   for any  $\zeta_-,\zeta_+\in H^1(-1,1)$ satisfying \eqref{eq:41},
\begin{equation} 
\label{eq:47}
\|\zeta_\pm \|_{1,2}\leq C_0\,.
\end{equation}
 for any pair $(v,g)$ in $ D(\LL_\beta^\zeta) \times H^1(-1,1) $ satisfying 
  \begin{equation}
    \label{eq:48}
  (\LL_\beta^\zeta-\beta\lambda)v=g\,,
  \end{equation}
  we have 
\begin{equation}
\label{eq:49}
  |v(\pm1)|\leq C\,   [1+|\lambda_\pm \,  |\beta^{1/3}]^{1/2}\, \beta^{-2/3} \,  \|g\|_{1,2}\,.
\end{equation}
\end{proposition}
\begin{proof}
  Let (see \cite[Eq. (6.20)]{AH1} for a similar construction)
\begin{equation}
\label{eq:50}
    v_D=v -v(1)\hat{\psi}_+ - v(-1)\hat{\psi}_- \,.
\end{equation}
Then $v_D$ is in the domain  of $\LL_\beta^D$ 
and satisfies
\begin{subequations}
  \label{eq:51}
  \begin{equation}
(\LL_\beta^D-\beta\lambda)v_D=g-v(1)\hat{g}_+ - v(-1)\hat{g}_- \,,
\end{equation}
where
\begin{equation}
    \hat{g}_\pm  =    (-\frac{d^2}{dx^2} +i\beta(U+i\lambda))\hat{\psi}_\pm 
\end{equation}
\end{subequations}
Recall from  \cite[Lemma 6.1]{AH1}
    \begin{equation} 
\label{eq:52}
\hat{g}_\pm = i\beta\, [U- U(\pm 1) -J_\pm(x\mp 1))]\, \hat{\psi}_\pm \mbox{ in }
      (-1,1)\,,
\end{equation}
and hence \cite[Eq. (8.96)]{AH1}
\begin{equation}
\label{eq:53}
  \| \hat g_\pm \|_2\leq C\, \beta^{1/6} [1+|\lambda_\pm|\beta^{1/3}]^{-5/4}\,. 
\end{equation}

We now use \cite[Lemma 5.7]{AH1} to obtain that 
\begin{equation}
\label{eq:54}
  \|v_D-w\|_1\leq C\beta^{-2/3}\big([1+|\lambda_+|\beta^{1/3}]^{-5/4}|v(1)|+ [1+|\lambda_-|\beta^{1/3}]^{-5/4}|v(-1)|\big) \,,
\end{equation}
where
\begin{displaymath}
  w=(\LL_\beta^D-\beta\lambda)^{-1}g\,.
\end{displaymath}
For the estimate of $\|w\|_1$  we use 
\eqref{eq:38} to obtain the following decomposition
\begin{subequations}
\label{eq:55}
  \begin{equation}
  w= i\frac{g(x_\nu)}{\beta[U-\nu- i \max(-\mu,\beta^{-1/3})]}+\tilde{w} \,,
\end{equation}
where
\begin{equation}
  \|\tilde{w}\|_1\leq C \beta^{-1}\, \|g\|_{1,2} \,.
\end{equation}
\end{subequations}
Since $v\in D(\LL_\beta^\zeta)$ we can write
\begin{equation}
\label{eq:56}
  0=\langle v,\zeta_\pm\rangle=\langle v_D-w,\zeta_\pm\rangle+\langle w,\zeta_\pm\rangle+v(1)\langle{ \hat{\psi}_+},\zeta_\pm\rangle+v(-1)\langle{ \hat{\psi}_-},\zeta_\pm\rangle
\end{equation}
Then we have by \eqref{eq:54}   and \eqref{eq:55}
\begin{multline}
\label{eq:57}
  |\langle v_D,\zeta_\pm\rangle|\leq |\langle w,\zeta_\pm\rangle|
   \\ + C\beta^{-2/3}\big([1+|\lambda_+|\beta^{1/3}]^{-5/4}|v(1)|+
  [1+|\lambda_-|\beta^{1/3}]^{-5/4}|v(-1)|\big)\,\|\zeta_\pm\|_\infty  \,.
\end{multline}
By (\ref{eq:55}a) it holds that
\begin{displaymath}
 |\langle w,\zeta_\pm\rangle|\leq   \Big|\Big\langle \frac{g(x_\nu)}{\beta[U-\nu- i
    \max(-\mu,\beta^{-1/3})]},\zeta_\pm\Big\rangle\Big|+ \|\tilde{w}\|_1\,\|\zeta_\pm\|_\infty \,,
\end{displaymath}
 which implies by (\ref{eq:55}b) and integration by parts  that
\begin{equation}
\label{eq:58}
\begin{array}{ll}
   |\langle w,\zeta_\pm\rangle| &  \leq\frac{|g(x_\nu)|}{\beta}\bigg|\frac{\zeta_\pm}{U^\prime}\log([U-\nu- i
    \max(-\mu,\beta^{-1/3})])\Big|_{-1}^1\bigg|\\ 
    & \quad +  \frac{\|g\|_\infty}{\beta}\,\Big|\Big\langle \log[U-\nu- i
    \max(-\mu,\beta^{-1/3})],\Big(\frac{\zeta_\pm}{U^\prime}\Big)^\prime\Big\rangle\Big|+\frac{C}{\beta}\|g\|_{1,2}\,.
    \end{array}
\end{equation}
Since $g\in H^1_0(-1,1)$,  we may write,  for $\mu\geq-1$,
\begin{displaymath}
\begin{array}{l}
  |g(x_\nu)|\bigg|\frac{\zeta_\pm}{U^\prime}\log([U-\nu- i\max(-\mu,\beta^{-1/3})])\Big|_{-1}^1 \bigg|\\  \quad \leq
   C \, \|g^\prime\|_2|\min(1-x_\nu,1+x_\nu)|^{1/2} [|\log(U(1)-\nu)|+ |\log(U(-1)-\nu)| ]\\
  \qquad \leq
  \widehat C\, \|g^\prime\|_2\,.
  \end{array}
\end{displaymath}
From the above, \eqref{eq:58}, and \eqref{eq:47} we obtain that
\begin{equation}
\label{eq:59}
   |\langle w,\zeta_\pm\rangle|\leq \frac{C}{\beta}\, \|g\|_{1,2}\,.
\end{equation}
 Note that for $\mu<-1$ \eqref{eq:59} easily follows from
  \eqref{eq:55}. \\
    From   \eqref{eq:57} and \eqref{eq:59}, we get
  \begin{equation}
\label{eq:52b}
   |\langle v_D,\zeta_\pm\rangle|\leq  C\beta^{-1} \|g\|_{1,2} + C\beta^{-2/3}\big([1+|\lambda_+|\beta^{1/3}]^{-5/4}|v(1)|+
  [1+|\lambda_-|\beta^{1/3}]^{-5/4}|v(-1)|\big)  \,.
\end{equation}
Then we may proceed as in the proof of \cite[Lemma 6.2]{AH1} to
obtain, using \cite[Eq. (8.91)]{AH1}, that
  \begin{equation}
\label{eq:60}
  |\langle\zeta_\pm -1, \hat{\psi}_\pm \rangle| + |\langle\zeta_\pm , \hat{\psi}_\mp\rangle| \leq 
  2\|\zeta^\prime_\pm \|_2\|[1\mp x]^{1/2}\hat{\psi}_\pm \|_1\leq C\beta^{-1/2}[1+|\lambda_\pm|\beta^{1/3}]^{-3/4} \,.
\end{equation}
 Returning to \eqref{eq:56},   we rewrite it  (for one choice of $\pm$)  in the form
\begin{equation} \label{eq:61}
\begin{array}{ll}
  \langle \hat{\psi}_+,1\rangle v(1)& =  \langle \hat{\psi}_+,\zeta_+ \rangle v(1) + \langle \hat{\psi}_+,1-\zeta_+ \rangle v(1)\\
  &= -  \langle v_D,\zeta_+\rangle   - v(-1)\langle{ \hat{\psi}_-},\zeta_+\rangle  + \langle \hat{\psi}_+,1-\zeta_+ \rangle v(1)
 \end{array}
  \end{equation}
Then we get from \eqref{eq:60} and \eqref{eq:61}
\begin{equation}\label{eq:62}
|v(1)\langle \hat{\psi}_+,1 \rangle|  \leq C\beta^{-1/2}[1+|\lambda_+|\beta^{1/3}]^{-3/4}(|v(1)|+|v(-1)|) +   |\langle v_D,\zeta_+\rangle|\,.
\end{equation}
  Observing  that 
\begin{displaymath}
  \langle \hat{\psi}_\pm,1\rangle=C_\pm\beta^{-1/3}[1+|\lambda_\pm|\beta^{1/3}]^{-1/2}[1+\OO(\beta^{-1})] \,,
\end{displaymath}
 for some $C_\pm\neq0$ (which is guaranteed for $\Upsilon<\vartheta_1^r$), 
 we  obtain that
\begin{displaymath}
  |v(1)|\leq C\beta^{-1/6}[1+|\lambda_+|\beta^{1/3}]^{-1/4}(|v(1)|+|v(-1)|)+ C\beta^{1/3}[1+|\lambda_+|\beta^{1/3}]^{1/2}|\langle v_D,\zeta_\pm\rangle|\,.
\end{displaymath}
A similar estimate for $v(-1)$ can be obtained as well.\\ 

The above, together with \eqref{eq:52b},  yields
\eqref{eq:49}.
\end{proof}

\section{Orr-Sommerfeld estimates}
\label{sec:4}

\subsection{Preliminaries}
For the convenience of the reader we repeat here the definition of the
Orr-Sommerfeld operator from \eqref{eq:6}
\begin{subequations}
  \label{eq:63}
  \begin{equation}
 B_{\lambda,\alpha,\beta}^\D \phi= \Big(-\frac{d^2}{dx^2}+i\beta(U+i\lambda)\Big)(-\phi^{\prime\prime}+\alpha^2\phi)+i\beta U^{\prime\prime}\phi \,,
\end{equation}
where $\phi\in D(B_{\lambda,\alpha,\beta}^D)$ and hence must satisfy
\begin{equation}
  \phi(\pm1)=\phi^\prime(\pm1)=0 \,.
\end{equation}
\end{subequations}
Let $\phi\in D(\B_{\lambda,\alpha,\beta}^\D)$, $f= \B_{\lambda,\alpha,\beta}\, \phi\in L^2(-1,1)$
and  $v_D\in H^2(-1,+1)$ be defined by
 \begin{equation}
  \label{eq:64}
  v_D=(U+i\lambda)(-\phi^{\prime\prime}+\alpha^2\phi)+U^{\prime\prime}\phi+ (U+i\lambda)[\phi^{\prime\prime}(1)\hat{\psi}_+ + \phi^{\prime\prime}(-1)\hat{\psi}_-] \,.
\end{equation}
It can be easily verified that 
\begin{displaymath}
  v_D \in  H^2(-1,+1)\cap H^1_0(-1,+1)\,.
\end{displaymath}
A simple computation  (see \cite[Eq. (7.4)-(7.5)]{AH1}) yields
that
\begin{subequations}
  \label{eq:65}
  \begin{equation}
   \Big(-\frac{d^2}{dx^2}+i\beta(U+i\lambda)\Big)v_D=g_D \,,
\end{equation}
where
\begin{equation}
\noindent  g_D=(U+i\lambda)( f+\phi^{\prime\prime}(1)\hat{g}_+ +\phi^{\prime\prime}(-1)\hat{g}_-) -   (U^{\prime\prime}\phi)^{\prime\prime} 
  -2U^\prime \tilde{v}_D^\prime-U^{\prime\prime} \tilde{v}_D\,,
\end{equation}
 $ \hat{g}_\pm$ is given by \eqref{eq:52} and
\begin{equation}
  \tilde{v}_D=-\phi^{\prime\prime}+\alpha^2\phi+\phi^{\prime\prime}(1)\hat{\psi}_+ +\phi^{\prime\prime}(-1)\hat{\psi}_- \,.
\end{equation}
\end{subequations}

We begin by estimating the contribution of the boundary terms in
\eqref{eq:64}. 
\begin{lemma}
\label{lem:inviscid-decay} 
 Let $U\in C^3([-1,1])$ satisfy \eqref{eq:5}. Let further $J_\pm=U^\prime(\pm1)$.
There exist positive constants $C$, $\mu_0$, and $\beta_0$ such that, for
all $\beta\geq \beta_0$ and $\lambda=\mu+i \nu$ s.t.  $|\mu|\leq \mu_0$,
it holds that
   \begin{equation}
\label{eq:66}
 \|(\A_{\lambda,\alpha}^\D)^{-1}(U+i\lambda)\hat{\psi}_\pm \|_{1,2} \leq 
 C\, [1+\beta^{1/3}|\lambda_\pm|]^{-3/4} \, \beta^{-1/2}
 \,,
   \end{equation}
where 
\begin{equation}
  \label{eq:67}
\lambda_\pm=\mu + i(\nu-U(\pm1)) \,.
\end{equation}
 \end{lemma}
 \begin{proof}
  Since $|\mu|\leq \mu_0$ we may use  Proposition
    \ref{cor:strip-bounded-inverse}  to obtain the invertibility of $
    A_{\lambda,\alpha}^\D$ and then use \eqref{eq:28} to obtain
  \begin{displaymath}
     \|(\A_{\lambda,\alpha}^\D)^{-1}(U+i\lambda)\hat{\psi}_\pm \|_{1,2} \leq C\, \|(1\mp x)^\frac 12 \hat \psi_\pm\|_1 \,,
  \end{displaymath}
  from which the lemma easily follows using \cite[Eq. (8.91)]{AH1}.
 \end{proof}

\subsection{Bounded $\alpha$ and $|\lambda|$}
\label{sec:3.1}

In this subsection we prove the following result:
\begin{proposition}
\label{lem:no-slip-convex-U}
Let $\lambda_0$  and $\alpha_0$ be positive.  There exist positive $\beta_0$,
$\Upsilon$, $\mu_0$, and $C>0$ such that for all $\beta\geq \beta_0$,  $\lambda =\mu + i \nu$
s.t. $-\mu_0\leq \mu \leq \Upsilon\beta^{-1/3}$, $|\nu|\leq \lambda_0$,  $\alpha \in [0,\alpha_0]$,
$\B_{\lambda,\alpha,\beta}^\D$ is invertible and it holds that
\begin{subequations}
\label{eq:68}
    \begin{equation}
      \sup_{
        \begin{subarray}{c}
       -\mu_0\leq\Re\lambda\leq\Upsilon\beta^{-1/3} \\
          0\leq \alpha\leq \alpha_0
        \end{subarray}}\big\|(\B_{\lambda,\alpha,\beta}^\D)^{-1}\big\|+
      \big\|\frac{d}{dx}\, (\B_{\lambda,\alpha,\beta}^\D)^{-1}\big\|\leq
      C \, \beta^{-1/2}\,.
  \end{equation}
Furthermore, we have 
    \begin{equation}
      \sup_{
        \begin{subarray}{c}
          \Re\lambda=\Upsilon\beta^{-1/3} \\
          0\leq \alpha\leq \alpha_0
        \end{subarray}}\big\|(\B_{\lambda,\alpha,\beta}^\D)^{-1}\big\|+
      \big\|\frac{d}{dx}\, (\B_{\lambda,\alpha,\beta}^\D)^{-1}\big\|\leq C\,  \beta^{-5/6 }\,.
  \end{equation}
\end{subequations}
\end{proposition}

To prove the above result we need to prove first the following auxiliary
estimates 
\begin{lemma}
\label{lem:aux}
  Let $\alpha_0$ and $\lambda_0$ be positive numbers. There exist positive $C$,
  $\beta_0$, and $\Upsilon$ such that,  for all $\beta\geq \beta_0$,  
  $\lambda =\mu + i \nu$ s.t. $-\lambda_0<\mu<\Upsilon\beta^{-1/3}$ and  $|\nu|\leq \lambda_0$, $\alpha \in
  [0,\alpha_0]$,  and $\phi \in  H^4(-1,1)\cap H^2_0(-1,1)$  it holds that 
  \begin{equation}
    \label{eq:69}  
\|g_D-(U+i\lambda)( f+\phi^{\prime\prime}(1)\hat{g}_+ +\phi^{\prime\prime}(-1)\hat{g}_-) \|_2\leq
C\,\big(\sigma\|\phi^\prime\|_2+\beta^{1/6}\sigma^{-1}\|f\|_2\big) \,,
  \end{equation}
where
\begin{equation}
  \label{eq:70}
\sigma(\beta,\mu)=\mu_+^{1/2}\beta^{2/3}+\beta^{1/3}+|\mu|^{1/2}\beta^{1/2}\log^{1/2}\beta
\end{equation}
  where $f$ is given in \eqref{eq:1.7}, $g_D$ is given in
  (\ref{eq:65}b), and, with $\lambda_m=\min(|\lambda_+|,|\lambda_-|)$,
\begin{equation}
 \label{eq:71}
\|\phi^{\prime\prime}(1)\hat{g}_+ +\phi^{\prime\prime}(-1)\hat{g}_-\|_2\leq C\, 
 \big[1+\lambda_m\beta^{1/3}]^{-1/4}[\beta^{1/2}\|\phi^\prime\|_2+\beta^{-1/3}\|f\|_2\big] \,.
\end{equation}
\end{lemma}
\begin{proof}
To prove \eqref{eq:69} we need an  estimate for $g_D$. We begin with
an estimate of the first term on the right-hand-side of the definition
of $g_D$ in (\ref{eq:65}b) yielding, thereby, a proof of \eqref{eq:71}.

{\em Step 1:} Prove \eqref{eq:71}. 

Recall that $\hat{g}_\pm $ is given by \eqref{eq:52}.
We seek first an estimate for $\phi^{\prime\prime}(\pm1)$ by using Lemma
\ref{lem:boundary}. To this end, we introduce, for any $\alpha$, the functions  $\zeta_+,\zeta_-\in
H^2(-1,1)$ satisfy
  \begin{equation}
\label{eq:130}
    \begin{cases}
      -\zeta_\pm^{\prime\prime}+\alpha^2\zeta_\pm=0 & x\in(-1,1) \\
      \zeta_\pm(\pm1)=1 & \zeta_\pm(\mp1)=0\,.  
    \end{cases}
  \end{equation}
   Note that since $0\leq\alpha\leq\alpha_0$ it holds that $\|\zeta_\pm^\prime\|_2\leq C$ as is
    required in the statement of Lemma \ref{lem:boundary}.
  It can be easily verified that for any $\phi\in H^2_0(-1,1)$ it holds that 
\begin{equation}
\label{eq:72}
  \langle\zeta_\pm,-\phi^{\prime\prime}+\alpha^2\phi\rangle=0 \,.
\end{equation}
Hence we may write \eqref{eq:63} in the form
\begin{equation}
\label{eq:73}
  (\LL_\beta^\zeta-\lambda)v=f-i\beta U^{\prime\prime}\phi\,,
\end{equation}
where $v=-\phi^{\prime\prime}+\alpha^2\phi$.  Consequently,  observing that $v(\pm
  1)=-\phi''(\pm 1)$, we may use \cite[Eq. (6.35)]{AH1}  (with
  $g=f$) 
and \eqref{eq:49}  (with
  $g=-i\beta U^{\prime\prime}\phi$) to obtain that
\begin{equation}
  \label{eq:74}
|\phi^{\prime\prime}(\pm1)|\leq C\, [1+|\lambda_\pm|\beta^{1/3}]^{1/2}[\beta^{1/3} \|\phi\|_{1,2} +\beta^{-1/2}\|f\|_2]  \,.
\end{equation}
Combining the above with \eqref{eq:53} yields
\begin{displaymath}
  \|\phi^{\prime\prime}(1)\hat{g}_+ +\phi^{\prime\prime}(-1)\hat{g}_-\|_2\leq C
 [1+\lambda_m\beta^{1/3}]^{-1/4}[\beta^{1/2} \|\phi\|_{1,2} +\beta^{-1/3}\|f\|_2] \,,
\end{displaymath}
which is precisely \eqref{eq:71}. 

In the next two steps, we obtain an estimate for the last two terms on the
right-hand-side of (\ref{eq:65}b). \\

{ {\em Step 2:} Estimate $\tilde{v}_D$.}

\noindent As $ \tilde{v}_D\in H^1_0(-1,1)$ we
may write, using (\ref{eq:65}c) and \eqref{eq:1.7},  
\begin{equation}
\label{eq:75}
  (\LL_\beta^D-\beta\lambda)\tilde{v}_D=f-i\beta U^{\prime\prime}\phi+\phi^{\prime\prime}(1)\hat{g}_+ +\phi^{\prime\prime}(-1)\hat{g}_-\,.
\end{equation}
We decompose $\tilde{v}_D$ in the following manner
\begin{subequations}
\label{eq:76}
\begin{equation}
  \tilde{v}_D=\tilde{v}_D^1+\tilde{v}_D^2+\tilde{v}_D^3\,,
\end{equation}
where 
\begin{equation}
  \begin{array}{lll}
&\tilde{v}_D^1 & =-i\beta(\LL_\beta^D-\beta\lambda)^{-1} ( [U^{\prime\prime}-U^{\prime\prime}(x_\nu)] \phi)\\
  &\tilde{v}_D^2& =(\LL_\beta^D-\beta\lambda)^{-1}[\phi^{\prime\prime}(1)\hat{g}_+ +\phi^{\prime\prime}(-1)\hat{g}_-+f]\,,\\
 &\tilde{v}_D^3& =-i\beta(\LL_\beta^D-\beta\lambda)^{-1} U^{\prime\prime}(x_\nu)\phi\,. 
\end{array}
\end{equation}
\end{subequations}

By applying \eqref{eq:39} with $f= \Big(\frac{
  U^{\prime\prime}-U^{\prime\prime}(x_\nu)}{U-\nu}\Big) \phi$, we obtain  that 
\begin{equation}
\label{eq:77} 
  \|\tilde{v}_D^1\|_2\leq C\, \Big\| \Big(\frac{ U^{\prime\prime}-U^{\prime\prime}(x_\nu)}{U-\nu}\Big) \phi  \Big\|_2 \leq \widehat C\,  \| \phi\|_2 \,.
\end{equation}
Using \eqref{eq:37} and \cite[Lemma 5.7]{AH1} yields
\begin{displaymath}
 \|\tilde{v}_D^2\|_2+\beta^{-1/3}\|(\tilde{v}_D^2)^\prime\|_2 +\beta^{1/6}\|\tilde{v}_D^2\|_1\leq C\, \beta^{-2/3}[|\phi^{\prime\prime}(1)|\,\|\hat{g}_+\|_2
  +|\phi^{\prime\prime}(-1)|\,\|\hat{g}_-\|_2+\|f\|_2]\,.
\end{displaymath}
By \eqref{eq:71}  we then obtain that
\begin{multline}
  \label{eq:78}
\|\tilde{v}_D^2\|_2+\beta^{-1/3}\|(\tilde{v}_D^2)^\prime\|_2
+\beta^{1/6}\|\tilde{v}_D^2\|_1 \\ \leq
C\, \big(\beta^{-1/6}[1+\lambda_m\beta^{1/3}]^{-1/4}\|\phi^\prime\|_2 +\beta^{-2/3}\|f\|_2\big)\,. 
\end{multline}
Finally, by \cite[Proposition 5.4]{AH1} it holds that
\begin{equation}
\label{eq:79}
    \|\tilde{v}_D^3\|_2\leq C\beta^{1/6}|U^{\prime\prime}(x_\nu)|\,\|\phi\|_\infty \,.
\end{equation}
Consequently, by the above, \eqref{eq:77}, and \eqref{eq:78},
\begin{equation}
  \label{eq:80}
  \|\tilde{v}_D\|_2\leq C\big( \beta^{1/6}\, \|\phi^\prime\|_2 +\beta^{-2/3}\|f\|_2\big) \,. 
\end{equation}

{\em Step 3:} Estimate $\tilde{v}_D^\prime$.\\

For the estimation of $g_D$ we need an estimate of
$\|\tilde{v}_D^\prime\|_2$ as well. \\
Note that if apply \eqref{eq:37} directly to (\ref{eq:76}b) we obtain
that
\begin{equation}
\label{eq:81}
  \| (\tilde{v}_D^1)'\|_2 \leq C\, \beta^{1/2} \|\phi\|_2\,,
\end{equation}
which is unsatisfactory. Consequently, we introduce the decomposition 
\begin{equation}
\label{eq:82}
  \tilde{v}_D^1= -\frac{U^{\prime\prime}-U^{\prime\prime}(x_\nu)}{U-\nu}\phi+u_D \,.
\end{equation}
Note that the first term on the right-hand-side satisfies
  \begin{displaymath}
    \Big\|\Big(\frac{U^{\prime\prime}-U^{\prime\prime}(x_\nu)}{U-\nu}\phi\Big)^\prime\Big\|_2\leq C\, \|\phi\|_{1,2}\,,
  \end{displaymath}
(compare with \eqref{eq:81}), whereas $u_D$ is a correction term satisfying
\begin{equation}
\label{eq:83}
    (\LL_\beta^D-\beta\lambda)u_D=-\Big(\frac{U^{\prime\prime}-U^{\prime\prime}(x_\nu)}{U-\nu}\phi\Big)^{\prime\prime}+i\beta\mu\frac{U^{\prime\prime}-U^{\prime\prime}(x_\nu)}{U-\nu}\phi \,. 
\end{equation}
Let $u_D=u_D^1+u_D^2$ where 
\begin{displaymath}
u_D^1=-(\LL_\beta^D-\beta\lambda)^{-1}\Big(\frac{U^{\prime\prime}-U^{\prime\prime}(x_\nu)}{U-\nu}\phi\Big)^{\prime\prime}\,.
\end{displaymath}
By \eqref{eq:37} and \cite[Lemma 5.7]{AH1} it holds that 
\begin{equation}\label{eq:84}
    \|u_D^1\|_2+  \beta^{1/6} \|u_D^1\|_1\leq C\beta^{-2/3}(\|\phi^{\prime\prime}\|_2+\|\phi^\prime\|_2)\,.
\end{equation}

To estimate $u_D^2=i\beta \mu
(\LL_\beta^D-\beta\lambda)^{-1}\big(\frac{U^{\prime\prime}-U^{\prime\prime}(x_\nu)}{U-\nu}\phi\big) $, we use
again \eqref{eq:37},  the trivial estimate
\begin{displaymath}   \|u_D^2\|_1\leq\sqrt{2}\,  \|u_D^2\|_2\,.
 \end{displaymath}  
  and  \cite[Lemma 5.7]{AH1} to obtain 
\begin{subequations}\label{eq:85}
 \begin{equation}  
   \|u_D^2\|_2\leq  C\, \|\phi\|_2\,,
\end{equation}
 and  
\begin{equation}
  \|u_D^2\|_1\leq C\min(|\mu|\log\beta,1)  \|\phi\|_\infty \,.
\end{equation}
\end{subequations}
Consequently, we deduce from \eqref{eq:84} and \eqref{eq:85}
\begin{subequations}
\label{eq:86}
  \begin{equation}
  \|u_D\|_2\leq C\, [\beta^{-2/3}(\|\phi^{\prime\prime}\|_2+\|\phi^\prime\|_2)+\|\phi\|_2]\,,
\end{equation}
and
\begin{equation}
  \|u_D\|_1\leq C[\min(|\mu|\,\log \beta,1) \|\phi\|_\infty+\beta^{-5/6}\, \|\phi^{\prime\prime}\|_2] \,.
\end{equation}
\end{subequations}
Returning to (\ref{eq:65}c) then yields 
\begin{displaymath}
  \|\phi^{\prime\prime}\|_2\leq \|\tilde{v}_D\|_2+ \alpha^2\|\phi\|_2 + \|\phi^{\prime\prime}(1)\hat{\psi}_+
  +\phi^{\prime\prime}(-1)\hat{\psi}_-\|_2 \,, 
\end{displaymath}
and since by \cite[Eqs. (6.17), (8.86), and (8.87)]{AH1} 
\begin{equation}
\label{eq:87}
  \|\hat{\psi}_\pm\|_2\leq C\beta^{-1/6}[1+|\lambda_\pm|\beta^{1/3}]^{-1/4}\,,
\end{equation}
we obtain by \eqref{eq:74} and \eqref{eq:80} that
\begin{equation}
\label{eq:88}
   \|\phi^{\prime\prime}\|_2\leq C\,\big[1+\lambda_M\beta^{1/3}\big]^{1/4}\,\big(\beta^{1/6}
   \|\phi^\prime\|_2+\beta^{-2/3}\|f\|_2\big)\,,
\end{equation}
where  
\begin{displaymath}
  \lambda_M:=\max(|\lambda_+|,|\lambda_-|)\leq 2\lambda_0+\max(|U(1)|,|U(-1)|)\,.
\end{displaymath}

We note that \eqref{eq:88} would prove unsatisfactory in the sequel
where the  term $[1+\lambda_M\beta^{1/3}]^{1/4}$ is multiplied by negative powers of
$(1+\lambda_\pm\beta^{1/3})$.  To address this problem we estimate in the same
manner, $ \|\phi^{\prime\prime}\|_{L^2(0,1)}$ and $\|\phi^{\prime\prime}\|_{L^2(-1,0)}$.  From
(\ref{eq:65}c) we then obtain that
\begin{displaymath}
   \|\phi^{\prime\prime}\|_{L^2(0,1)}\leq C\, [1+\lambda_+\beta^{1/3}]^{1/4}\big(\beta^{1/6}
   \|\phi^\prime\|_2+\beta^{-2/3}\|f\|_2\big)+|\phi^{\prime\prime}(-1)|\|\hat{\psi}_-\|_{L^2(0,1)}  \,,
\end{displaymath}
and since by  \cite[Eqs. (6.17), (8.86), and (8.87)]{AH1} 
\begin{displaymath}
  \|\hat{\psi}_-\|_{L^2(0,1)}\leq  \|(1+x)^4\hat{\psi}_-\|_{L^2(0,1)}\leq C\beta^{-3/2} [1+\lambda_-\beta^{1/3}]^{-9/4}\,,
\end{displaymath}
we may conclude, using \eqref{eq:74},  that 
\begin{equation}
\label{eq:89}
    \|\phi^{\prime\prime}\|_{L^2(0,1)}\leq C\, [1+\lambda_+\beta^{1/3}]^{1/4}\big(\beta^{1/6}
   \|\phi^\prime\|_2+\beta^{-2/3}\|f\|_2\big)\,.
\end{equation}
In a similar fashion we obtain that
\begin{equation}
  \label{eq:90}
 \|\phi^{\prime\prime}\|_{L^2(-1,0)}\leq C[1+\lambda_-\beta^{1/3}]^{1/4}\big(\beta^{1/6}
   \|\phi^\prime\|_2+\beta^{-2/3}\|f\|_2\big)\,.
\end{equation}
Substituting \eqref{eq:88} into \eqref{eq:86} yields, using  Sobolev
embeddings and the fact that $\phi(1)=0$, 
\begin{subequations}
\label{eq:91}
  \begin{equation}
  \|u_D\|_2\leq
  C\, [\|\phi^\prime\|_2+\beta^{-5/4}\|f\|_2] \,.
\end{equation}
and 
\begin{equation}
  \|u_D\|_1\leq
  C\, [\min(|\mu|\,\log \beta,1)\|\phi^\prime\|_2+  \beta^{-17/12}\|f\|_2] \,.
\end{equation}
\end{subequations}

Taking the inner product of \eqref{eq:83} with $u_D$ gives after
integration by parts
\begin{displaymath}
  \|u_D^\prime\|_2^2 =\mu\beta\|u_D\|_2^2+
 \Re\Big\langle u_D^\prime,\Big(\frac{U^{\prime\prime}-U^{\prime\prime}(x_\nu)}{U-\nu}\phi\Big)^\prime\Big\rangle
 +\beta\mu\, \Im\Big\langle u_D,\frac{U^{\prime\prime}-U^{\prime\prime}(x_\nu)}{U-\nu}\phi\Big\rangle \,.
\end{displaymath}
From here we conclude (separately addressing the cases $\mu>0$ and $\mu\leq 0$)   that
\begin{displaymath}
  \|u_D^\prime\|_2^2 \leq C\, \Big(\mu_+\beta\|u_D\|_2^2+\|\phi^\prime\|_2^2+\mu_+\beta\|u_D\|_1\|\phi\|_\infty  +\mu_-\beta\|\phi\|_2^2 \Big) \,.
\end{displaymath}
By \eqref{eq:91} we then have
\begin{displaymath}
  \|u_D^\prime\|_2\leq C\,\big[(|\mu|^{1/2}\beta^{1/2}+1)\|\phi^\prime\|_2+\beta^{-3/4}\|f\|_2\big]  \,.
\end{displaymath}
Combining the above with \eqref{eq:82} yields
\begin{equation}
  \label{eq:92}
\|(\tilde{v}_D^1)^\prime\|_2\leq C\,\big[(|\mu|^{1/2}\beta^{1/2}+1)\|\phi^\prime\|_2+\beta^{ -3/4}\|f\|_2\big] \,.
\end{equation}

We proceed with the estimate of $(\tilde{v}_D^3)^\prime$. 
 Since
\begin{displaymath} 
- \Re  \langle \tilde v_D^3, \beta U''(x_\nu) i\phi\rangle=   \Re\langle\tilde{v}_D^3,(\LL_\beta -\beta\lambda)\tilde{v}_D^3\rangle=\|(\tilde{v}_D^3)^\prime\|_2^2-\mu\beta \|\tilde{v}_D^3\|_2^2\,,
\end{displaymath}
we can conclude that
\begin{equation}
\label{eq:93}
  \|(\tilde{v}_D^3)^\prime\|_2^2=\mu\beta
 \|\tilde{v}_D^3\|_2^2 - \beta U^{\prime\prime}(x_\nu)\Re\langle\tilde{v}_D^3,i\phi\rangle
\end{equation}
To estimate the last term on the right-hand-side we write (see \eqref{eq:76})
\begin{equation}
\label{eq:94}
  \Re\langle\tilde{v}_D^3,i\phi\rangle= \Re\langle\tilde{v}_D,i\phi\rangle-\Re\langle\tilde{v}_D^1,i\phi\rangle-\Re\langle\tilde{v}_D^2,i\phi\rangle
\end{equation}
To obtain a bound for the first term on the right-hand-side we use
(\ref{eq:65}c) to obtain
\begin{equation}
\label{eq:95}
  \Re\langle\tilde{v}_D,i\phi\rangle=\Re \langle \phi^{\prime\prime}(1)\hat{\psi}_+,i\phi\rangle +\Re \langle \phi^{\prime\prime}(-1)\hat{\psi}_-,i\phi\rangle\,.
  \end{equation}
To estimate the right-hand-side we
observe, as in  \cite[Proof of Lemma 8.8]{AH1}, that
\begin{displaymath} 
\phi(x) =\int_x^1(\xi-x)\phi^{\prime\prime}(\xi )\,d\xi\,,
\end{displaymath}
from which we get that 
\begin{displaymath}
  |\phi(x)| \leq \frac{1}{\sqrt{3}} (1-x)^{3/2}\|\phi^{\prime\prime}\|_2  \,,
\end{displaymath}
to obtain by \cite[Eq. (8.91)]{AH1} 
\begin{multline*} 
  |\langle \phi,\hat{\psi}_+\rangle|\leq\|\phi^{\prime\prime}\|_{L^2(0,1)}\|(1-x)^{3/2}\hat{\psi}_+\|_1+
  \|\phi^{\prime\prime}\|_{L^2(-1,0)}\|(1-x)^3\hat{\psi}_+\|_1 \\  \leq 
  C\Big(\beta^{-5/6}[1+|\lambda_+|\beta^{1/3}]^{-5/4}\|\phi^{\prime\prime}\|_{L^2(0,1)}+\beta^{-4/3}[1+|\lambda_+|\beta^{1/3}]^{-2}\|\phi^{\prime\prime}\|_{L^2(-1,0)}\Big)\,.
\end{multline*}
Combining the above with \eqref{eq:74},  \eqref{eq:89} and \eqref{eq:90}
yields 
\begin{displaymath}
|\phi^{\prime\prime}(1)\langle \phi,\hat{\psi}_+\rangle|\leq C\,\big[\beta^{-1/3}\|\phi^\prime\|_2^2+\beta^{-7/6}\|\phi^\prime\|_2\|f\|_2+\beta^{-2}\|f\|_2^2\big]\leq C\,\big[\beta^{-1/3}\|\phi^\prime\|_2^2+\beta^{-2}\|f\|_2^2\big] \,.
\end{displaymath}
A similar estimate can be obtained for $\phi^{\prime\prime}(-1)\langle
\phi,\hat{\psi}_-\rangle$. Hence, 
\begin{equation}
  \label{eq:96}
 |\phi^{\prime\prime}(1)\langle \phi,\hat{\psi}_+\rangle+\phi^{\prime\prime}(-1)\langle \phi,\hat{\psi}_-\rangle|\leq
 C\,\big[\beta^{-1/3}\|\phi^\prime\|_2^2+ \beta^{-2}\|f\|_2^2\big]  \,.
\end{equation}
Substituting the above into \eqref{eq:95} yields 
\begin{equation}
\label{eq:97} 
   |\Re \langle\tilde{v}_D,i\phi\rangle|\leq C\, \big[\beta^{-1/3}\|\phi^\prime\|_2^2+  \beta^{-2}\|f\|_2^2\big]  \,.
\end{equation}

Next, we estimate the second term $ \Re \langle\tilde{v}_D^1,i\phi\rangle$ on the right-hand-side of
\eqref{eq:94}.\\
 By \eqref{eq:82} we obtain that 
\begin{displaymath} 
| \Re \langle\tilde{v}_D^1,i\phi\rangle| = |\Re \langle u_D,i\phi\rangle|\leq\|u_D\|_1\,\|\phi\|_\infty \,.
\end{displaymath}
Hence we may conclude by (\ref{eq:91}b)  that \\
\begin{equation}
\label{eq:98} 
   |\Re \langle\tilde{v}_D^1,i\phi\rangle|\leq
   C\big([\min(|\mu|\log\beta,1)+\beta^{-1/3}]\|\phi^\prime\|_2^2+\beta^{-5/2}\|f\|_2^2\big)\,.
\end{equation}

Finally, for the last term $\Re\langle\tilde{v}_D^2,i\phi\rangle$ on the
right-hand-side of 
\eqref{eq:94} we obtain  by \eqref{eq:78}  
\begin{equation}
\label{eq:99}
  |\Re\langle\tilde{v}_D^2,i\phi\rangle|\leq \|\tilde{v}_D^2\|_1\|\phi\|_\infty \leq
  C\, [\beta^{-1/3}\|\phi^\prime\|_2^2+  \beta^{-5/6}\|f\|_2\|\,\|\phi^\prime\|_2]\,. 
\end{equation}
Substituting the above, together with \eqref{eq:97} and \eqref{eq:98}
into \eqref{eq:94} yields
\begin{multline}\label{eq:100}
   |\Re\langle\tilde{v}_D^3,i\phi\rangle|\leq C\Big([\beta^{-1/3}+\min(|\mu|\log\beta,1)]\|\phi^\prime\|_2^2\\
   +
   \beta^{-5/6}\|f\|_2\|\,\|\phi^\prime\|_2+\beta^{-2}\|f\|_2^2\Big) 
\end{multline}
Substituting \eqref{eq:100} into \eqref{eq:93} 
together with \eqref{eq:79} leads to:
\begin{multline}
\label{eq:101}  
  \|(\tilde{v}_D^3)^\prime\|_2 \leq
  C\,\Big[( \mu_+^{1/2}\beta^{2/3}+\beta^{1/3}+|\mu|^{1/2}\beta^{1/2}
    \log^{1/2}\beta)\|\phi^\prime\|_2 \\ +   \beta^{1/12} \|f\|_2^{1/2}\,\|\phi^\prime\|_2^{1/2}+\beta^{-1/2}\|f\|_2\Big]\,.
\end{multline}
 Combining \eqref{eq:78},   and
\eqref{eq:101} yields
\begin{equation}
\label{eq:102}
   \|\tilde{v}_D^\prime\|_2 \leq
   C\,\big( \sigma\|\phi^\prime\|_2+\beta^{1/6}\sigma^{-1}\|f\|_2\big) \,,
\end{equation}
where, as in \eqref{eq:70},
$
\sigma(\beta,\mu)=\mu_+^{1/2}\beta^{2/3}+\beta^{1/3}+|\mu|^{1/2}\beta^{1/2}\log^{1/2}\beta\,.
$
\\

{\em Step 4:} Prove \eqref{eq:69}. \\

To complete the proof we need yet an estimate for $(U^{\prime\prime}\phi)^{\prime\prime}$. To
this end we use \eqref{eq:88} to obtain
\begin{equation}
  \label{eq:103}
  \|(U^{\prime\prime}\phi)^{\prime\prime}\|_2\leq
  C\big(\beta^{1/4}\|\phi^\prime\|_2+\beta^{ -7/12}\|f\|_2\big) \,. 
\end{equation}
Combining the above with \eqref{eq:102}, \eqref{eq:80}, 
and (\ref{eq:65}b) yields
\begin{displaymath}
  \|g_D-(U+i\lambda)( f+\phi^{\prime\prime}(1)\hat{g}_+ +\phi^{\prime\prime}(-1)\hat{g}_-)\|_2\leq
   C\, \big( \sigma\|\phi^\prime\|_2+\beta^{1/6}\sigma^{-1}\|f\|_2\big) \,.
\end{displaymath}
which is precisely \eqref{eq:69}. 
\end{proof}

For later reference we need an estimate of $v_D$. 
 \begin{remark}
By \eqref{eq:37}  and \eqref{eq:69} it
holds, under the same conditions as in  Lemma~\ref{lem:aux}, that
\begin{displaymath}
\|(\LL_\beta^D-\beta\lambda)^{-1}\big(g_D-(U+i\lambda)[f+\phi^{\prime\prime}(1)\hat{g}_+
+\phi^{\prime\prime}(-1)\hat{g}_-]\big)\|_2 \leq
C\, \big(\sigma\beta^{-2/3}\|\phi^\prime\|_2+ \beta^{-1/2}\sigma^{-1}\|f\|_2\big) \,, 
\end{displaymath}
with $\sigma$ introduced in \eqref{eq:70}.\\
By \eqref{eq:37}, \eqref{eq:39} 
and \eqref{eq:71}, we have
\begin{multline*}
  \|(\LL_\beta^D-\beta\lambda)^{-1}\big((U+i\lambda )[f+\phi^{\prime\prime}(1)\hat{g}_+
+\phi^{\prime\prime}(-1)\hat{g}_-]\big)\|_2 \\ \leq C\big([1+\lambda_m\beta^{1/3}]^{-1/4}\beta^{-1/2}
  \|\phi^\prime\|_2 +\beta^{-1}\|f\|_2\big) \,.
\end{multline*}
Consequently, we have that
\begin{equation}
  \label{eq:104}
\|v_D\|_2\leq C \, \big(\sigma\beta^{-2/3}\|\phi^\prime\|_2+\max(\beta^{-1},\beta^{-1/2}\sigma^{-1})\|f\|_2\big) \,.  
\end{equation}   
\end{remark}
~\\
We can now proceed to prove the main result of this subsection.
\begin{proof}[Proof of Proposition \ref{lem:no-slip-convex-U}]
  The proof is very similar to the last step of the proof of
  \cite[Lemma 8.8]{AH1}. \\

  {\em Step 1:} Prove \eqref{eq:68} for $-\mu_0\leq \mu \leq \Upsilon\beta^{-1/3}$, 
    where $\mu_0>0$ is the same as in the statement of Proposition
    \ref{cor:strip-bounded-inverse}, and $|\mu|\geq [\delta\beta]^{-1}$ for some
    $0<\delta<1$.\\ 

\noindent Using the notation introduced in
  \eqref{eq:11}  and \eqref{eq:64}, we begin with the
 decomposition 
\begin{subequations}\label{eq:105}
\begin{equation}
  \phi = \phi_D + \check{\phi}_+ + \check{\phi}_-\,,
\end{equation}
 where 
\begin{equation}
  \phi_D =(\A_{\lambda,\alpha}^\D)^{-1}v_D \quad \mbox{ and } \quad
 \check{\phi}_\pm =-(\A_{\lambda,\alpha}^\D)^{-1}\big([U+i\lambda]\phi^{\prime\prime}(\pm 1)\hat{\psi}_\pm \big) \,. 
\end{equation}
\end{subequations}
By  \eqref{eq:66} and \eqref{eq:74} we have
\begin{displaymath}
  \|\check{\phi}_\pm \|_{1,2}\leq C\, [1+|\lambda_\pm|\beta^{1/3}]^{ -1/4}\big(\beta^{-1/6}\|\phi\|_{1,2} + \beta^{-1}\|f\|_2\big)\,.
\end{displaymath}

Hence,  using (\ref{eq:105}a),  we get for sufficiently large $\beta$
\begin{equation}
  \label{eq:106}
 \|\check{\phi}_\pm \|_{1,2}\leq C[1+|\lambda_\pm|\beta^{1/3}]^{ -1/4}\big(\beta^{-1/6}\|\phi_D\|_{1,2}+ \beta^{-1}\|f\|_2\big)\,.
\end{equation}
Let
\begin{displaymath}
  g_D^1= g_D-(U+i\lambda)( f+\phi^{\prime\prime}(1)\hat{g}_+ +\phi^{\prime\prime}(-1)\hat{g}_-)\,,
\end{displaymath}
and 
\begin{displaymath}
  g_D^2=\phi^{\prime\prime}(1)\hat{g}_+ +\phi^{\prime\prime}(-1)\hat{g}_-+f \,.
\end{displaymath}
Note that
\begin{equation}\label{eq:106a}
  g_D=g_D^1+(U+i\lambda)g_D^2
\end{equation}
 Similarly let 
\begin{displaymath}
  v_D^1=(\LL_\beta^D-\beta\lambda)^{-1}g_D^1
\end{displaymath}
and
\begin{displaymath}
   v_D^2=(\LL_\beta^D-\beta\lambda)^{-1}(U+i\lambda)g_D^2\,.
\end{displaymath}
Note that by \eqref{eq:65}  and \eqref{eq:106a}  
\begin{displaymath}
v_D=v_D^1+v_D^2\,.
\end{displaymath}

By \eqref{eq:69}, together with 
\eqref{eq:105} and \eqref{eq:106}, we obtain that
\begin{equation}
\label{eq:107}
  \|g_D^1\|_2\leq C\, \big(\sigma\|\phi^\prime_D\|_2+\beta^{1/6}\sigma^{-1}\|f\|_2\big) \,.
\end{equation}
Similarly, combining \eqref{eq:71} with
\eqref{eq:106}, yields   
\begin{equation}\label{eq:107a}
  \|g_D^2\|_2\leq C\, \big(
 [1+\lambda_m\beta^{1/3}]^{-1/4}\beta^{1/2}\|\phi_D^\prime\|_2+\|f\|_2\big) \,.
\end{equation}
We now use  Proposition \ref{cor:strip-bounded-inverse} for $v=v_D^j$ ($j=1,2$),
which  gives (see \eqref{eq:28}) 
\begin{displaymath}
  \| (\mathcal A^\D_{\lambda,\alpha})^{-1} v_D^j\|_{1,2} \leq C\,  \Big\|(1\pm\cdot)^{1/2} \frac{v_D^j}{U+ i \lambda}\Big\|_1\,.
\end{displaymath}
As is proved in \cite[Proposition 4.14]{AH1} we have,  in view of
  \eqref{eq:5}, for $j=1,2$ and  $p,q\in\R_+$ satisfying $\frac 1q +\frac 1p=1$,
\begin{displaymath}
 \Big\|(1\pm\cdot)^{1/2} \frac{v_D^j}{U+ i \lambda}\Big\|_1\leq \sqrt{2} \, \|v_D^j\|_p\, \Big\| \frac{1}{U+i\lambda}\Big\|_q\leq C_q\, |\mu|^{-1/p}\|v_D^j\|_p\,.
\end{displaymath}
Hence, we obtain, for $ \phi_D =(\A_{\lambda,\alpha}^\D)^{-1}(v_D^1+v_D^2)$,
that for any $p>1$, 
\begin{equation}
\label{eq:128}
\|\phi'_D\|_2 \leq C \, \big(|\mu|^{-1/p}\|v_D^1\|_p+|\mu|^{-1/2}\|v_D^2\|_2\big)\,.
\end{equation}

Substituting \eqref{eq:128} into \eqref{eq:107} we can conclude that
\begin{equation}
\label{eq:108}
    \|g_D^1\|_2\leq
    C\, \big(\sigma|\mu|^{-1/p}\|v_D^1\|_p  + \sigma |\mu|^{-1/2}\|v_D^2\|_2+\beta^{1/6}\sigma^{-1}\|f\|_2\big) \,.
\end{equation}
Similarly, as in the proof of \cite[Proposition 4.10]{AH1},  in view of
  \eqref{eq:5}, 
\begin{displaymath}
 \Big\|(1\pm\cdot)^{1/2} \frac{v_D}{U+ i \lambda}\Big\|_1\leq \sqrt{2}\,  \|v_D\|_\infty\, \Big\|
 \frac{1}{U+i\lambda}\Big\|_1\leq C \log (|\mu|^{-1}) \|v_D\|_\infty\,.
\end{displaymath}
Substituting into \eqref{eq:107a},  we obtain that
\begin{equation}
\label{eq:109}
   \|g_D^2\|_2\leq C
 \big([1+\lambda_m\beta^{1/3}]^{-1/4}\log( |\mu|^{-1})\,\beta^{1/2}\|v_D\|_\infty+\|f\|_2\big)\,.
\end{equation}

Using \eqref{eq:37}  (for the term $\mu (\LL_\beta^D-\beta\lambda)^{-1} g_D^2$) together
with \eqref{eq:39} yields 
\begin{displaymath}
  \|v_D^2\|_2+\beta^{-1/2}\|(v_D^2)^\prime\|_2\leq C \beta^{-1} \|g_D^2\|_2\,,
\end{displaymath}
from which we conclude, by interpolation, for $2\leq p\leq\infty$ that 
\begin{equation}
\label{eq:110}
  \|v_D^2\|_p\leq C \beta^{-\frac{3}{4}-\frac{1}{2p}}\|g_D^2\|_2\,.
\end{equation}
By   \cite[Corollary 5.3]{AH1} it holds,  for $ 2\leq p\leq\infty$,
that 
\begin{equation}
\label{eq:111}
  \|v_D^1\|_p\leq C\beta^{-\frac{3p+2}{6p}}\|g_D^1\|_2\,.
\end{equation}
Consequently, by (\ref{eq:108}), \eqref{eq:110}, and (\ref{eq:111}) we
have 
\begin{equation}
\label{eq:112}
   \|g_D^1\|_2\leq
    C\big[\sigma\big(|\mu|^{-1/p}\beta^{-\frac{3p+2}{6p}}\|g_D^1\|_2+ |\mu|^{-1/2}\beta^{-1}\|g_D^2\|_2 \big)+\beta^{1/6}\sigma^{-1}\|f\|_2\big]\,.
\end{equation}
Similarly, by \eqref{eq:109}, \eqref{eq:110}, and (\ref{eq:111}) it holds that
\begin{displaymath}
   \|g_D^2\|_2\leq C\,\big[\log|\mu|^{-1}(\|g_D^1\|_2+\beta^{-1/4}\|g_D^2\|_2)+\|f\|_2\big] \,,
\end{displaymath}
from which we conclude, recalling that $|\mu|\geq \beta^{-1}$ and choosing $\beta$ large enough,
that 
\begin{equation}
\label{eq:113}
    \|g_D^2\|_2\leq C\, \big[\log|\mu|^{-1}\|g_D^1\|_2+\|f\|_2\big] \,.
\end{equation}
Substituting \eqref{eq:113}  into \eqref{eq:112}, recalling again that
in this step $|\mu|\geq \beta^{-1}$ is assumed, we obtain for $2\leq p$,
\begin{equation}
\label{eq:129}
    \|g_D^1\|_2\leq  C\big[\sigma\big(|\mu|^{-1/p}\beta^{-\frac{3p+2}{6p}}+ 
      |\mu|^{-1/2}\log|\mu|^{-1/2}\beta^{-1}\big)\|g_D^1\|_2+\beta^{1/6}\sigma^{-1}\|f\|_2\big]\,. 
\end{equation}
Here we have used the inequality $\sigma^2  \leq C  |\mu|^{1/2} \beta^{\frac 76 }$\, to obtain
\eqref{eq:129}.

We now estimate the coefficient of $\|g_D^1\|_2$ on the right-hand-side
of \eqref{eq:129}.  We split the discussion into two cases, depending
on the sign of $\mu$.\\

{\bf The case $\mu >0$.\\}

By \eqref{eq:70}, for $\mu >0$, it holds that
$\sigma(\beta,\mu)\leq 2 \mu_+^{1/2}\beta^{2/3}+\beta^{1/3}$. Consequently, for sufficiently
large $\beta$,
\begin{displaymath}
  \sigma|\mu|^{-1/p}\big(\beta^{-\frac{3p+2}{6p}}+\log|\mu|^{-1}\beta^{-\frac{3p+2}{4p}}\big)\leq 3(\Upsilon^{\frac{p-2}{2p}}+|\mu|^{-1/p}\beta^{-\frac{p+2}{6p}})\,.
\end{displaymath}
Since for $|\mu|\geq [\delta \beta]^{-1}$ it holds that
$|\mu|^{-1/p}\beta^{-\frac{p+2}{6p}}\leq\delta^{1/p}$ for any $4\leq p $,  we obtain, for
sufficiently small $\Upsilon$ and $\delta$, by choosing some $ p\geq 4$ 
\begin{equation}
\label{eq:4.55}
    \|g_D^1\|_2\leq  C\beta^{1/6}\sigma^{-1}\|f\|_2\,.
\end{equation}

{\bf The case  $\mu <0$.\\}

For $- \lambda_0 \leq \mu \leq  -[\delta \beta]^{-1}$, we have by \eqref{eq:70}
  \begin{displaymath}
    \sigma = |\mu|^{1/2}\beta^{1/2}\log^{1/2}\beta+\beta^{1/3}\,,
  \end{displaymath}
and hence, for $ p\geq 4$, 
\begin{multline*}
   \sigma|\mu|^{-1/p}\big(\beta^{-\frac{3p+2}{6p}}
   +\log|\mu|^{-1}\beta^{-\frac{3p+2}{4p}}\big) \\ \leq 
   2\big(|\mu|^{1/2-1/p}\beta^{-\frac{1}{3p}}\log^{1/2}\beta
   +  \delta ^{1/p} \beta^{\frac{1}{p}+\frac 13  -\frac{1}{3p}-\frac{1}{2}} \big)\leq C\, ( \beta^{-\frac{1}{3p}}\log^{1/2}\beta + \delta^{1/p})\,.
 \end{multline*}
Hence, by choosing $\delta$ small and $\beta$ large,  \eqref{eq:4.55} is valid
for $\mu<0$ as well.  

Together with \eqref{eq:113} and the lower bound $\sigma \geq \beta^{1/3}$,
\eqref{eq:4.55} yields (again for sufficiently small $\Upsilon$ and $\delta$ and
$|\mu|\geq [\delta\beta]^{-1}$), 
\begin{equation}\label{eq:4.56}
     \|g_D^2\|_2\leq C\, \|f\|_2 \,.
\end{equation}
By \eqref{eq:110}  for $p=2$ and \eqref{eq:111} for $p=4$
together with \eqref{eq:4.55} and \eqref{eq:4.56}, we then obtain, for
$|\mu| \geq [\delta\beta]^{-1}$, that
\begin{displaymath} 
  \|v_D^1\|_4\leq C\beta^{-5/12}\sigma^{-1}\|f\|_2\quad ;\quad   \|v_D^2\|_2\leq C\beta^{-1}\|f\|_2\,.
\end{displaymath}
By \eqref{eq:128} for $p=4$ we then have
\begin{displaymath}
    \|\phi_D\|_{1,2}\leq C\big[ |\mu|^{-1/2}\beta^{-1} + |\mu|^{-1/4} \sigma^{-1} \beta^{-5/12}]\|f\|_2 \,,
\end{displaymath}
By \eqref{eq:70}, we have
$$
\sigma^{-1} \leq \mu_+^{-1/2}\beta^{-2/3} \mbox{ and } \sigma^{-1} \leq  \beta^{-1/3}\,.
$$ 
Hence we get 
\begin{displaymath}
    \|\phi_D\|_{1,2}\leq C\,\big[ |\mu|^{-1/2}\beta^{-1}+
    \min( \mu_+^{-3/4}\beta^{-13/12},|\mu|^{-1/4}\beta^{-3/4})\big]\,\|f\|_2 \,.
\end{displaymath}\
By \eqref{eq:105} and  \eqref{eq:106}, this implies
\begin{equation}
  \label{eq:114}
  \|\phi\|_{1,2}\leq C\,\big[ |\mu|^{-1/2}\beta^{-1} +
    \min( \mu_+^{-3/4}\beta^{-13/12},|\mu|^{-1/4}\beta^{-3/4})\big]\, \|f\|_2 \,.
\end{equation}
Note that for $\mu=\Upsilon\beta^{-1/3}$ we obtain (\ref{eq:68}b).\\
From \eqref{eq:114}, we deduce that  for $|\mu|\geq [\delta\beta]^{-1}$
\begin{equation}
  \label{eq:114aa}
  \|\phi\|_{1,2}\leq C \beta^{-1/2}\|f\|_2 \,,
\end{equation}
which proves (\ref{eq:68}a) in this case.

{\em Step 2:} Prove \eqref{eq:68} for $|\mu|\leq [\delta\beta]^{-1}$.\\
 Let 
 \begin{displaymath}
   \tilde{\lambda}= 2[\delta \beta]^{-1}+i\nu :=\tilde \mu + i \nu \,,
 \end{displaymath}
 and then write
 \begin{displaymath}
   \B_{ \tilde{\lambda},\alpha}\phi=f+\big(2[\delta\beta]^{-1}-\mu\big)\beta (-\phi^{\prime\prime}+\alpha^2\phi)
   \,.
 \end{displaymath}
 Applying \eqref{eq:114aa} (which is applicable since $|\tilde
 \mu|\geq [\delta\beta]^{-1}$)
 to the above inequality
yields 
 \begin{equation}\label{eq:114a}
    \|\phi\|_{1,2}\leq C\, \big(\beta^{-1/2}\|f\|_2+\delta^{-1}\beta^{-1/2}\|-\phi^{\prime\prime}+\alpha^2\phi\|_2\big) \,. 
 \end{equation}
 By \eqref{eq:88} it holds that
\begin{displaymath}
   \|\phi\|_{1,2}\leq C\,\big(\beta^{-1/2}\|f\|_2+\delta^{-1}\beta^{-1/4}\|\phi\|_{1,2}\big) \,. 
\end{displaymath}
Combining the above with  \eqref{eq:114a} yields (\ref{eq:68}a) for
$|\mu|\leq [\delta\beta]^{-1}$. 
 \end{proof}

\subsection{The case $1\ll\alpha$}
\label{sec:2.1}

 While some of estimates in the previous subsection can be obtained
also for large values of $\alpha$ it is much more natural to study the case
separately, since the estimates become much simpler. 

\subsubsection{The case $1\ll\alpha\ll\beta^{1/3} $}

  For unbounded $\alpha$ we can no longer use \eqref{eq:49} and
  consequently also \eqref{eq:74}. However, since \eqref{eq:72} and
  \eqref{eq:73} remain valid, we may still use
  \cite[Eq. (6.35)]{AH1}, as long as $\alpha\beta^{-1/3}$ is
  sufficiently small.\\
More precisely, we have
  \begin{lemma} There exist positive $C$, $\alpha_1$, $\Upsilon$ and $\beta_0$ such that for any $\beta\geq \beta_0$,
$0\leq \alpha\leq \alpha_1\beta^{1/3}$,  and any $\lambda =\mu+ i \nu$ for which $\mu\leq \Upsilon\beta^{-1/3}$,
   we have, for $f= B_{\lambda,\alpha,\beta}^\D \phi$, 
  \begin{equation}
    \label{eq:115}
|\phi^{\prime\prime}(\pm1)|\leq C\,[1+|\lambda_\pm|\beta^{1/3}]^{1/2}\,\big[\beta^{1/3}\log\beta\, \|\phi\|_\infty +\beta^{-1/2}\|f\|_2\big]  \,.
  \end{equation}
  \end{lemma}
  \begin{proof}
    This result follows immediately from in \cite[Lemma 6.2 and Remark
    6.3]{AH1} which both hold under \cite[Eq. (6.22)]{AH1}.  It can be
    easily verified that  \cite[Eq. (6.22)]{AH1} is satisfied, for
    $\zeta_\pm$ given by \eqref{eq:130}, for sufficiently small
    $\alpha\beta^{-1/3}$. 
\end{proof}
Note that  \eqref{eq:74} (which follows from \eqref{eq:49}) and \eqref{eq:115}
differ by an additional $\log\beta$ on the right-hand-side, which is why
we preferred  \eqref{eq:74} in the previous subsection.

We can now state and prove the following result:
\begin{proposition}
  \label{prop:case-alpha-log}
There exist positive $C$, $\alpha_0$, $\alpha_1$, $\Upsilon$ and $\beta_0$ such that for any $\beta\geq \beta_0$,
$\alpha_0\leq \alpha\leq \alpha_1\beta^{1/3}$,  and any $\lambda =\mu+ i \nu$ such that $\mu\leq \Upsilon\beta^{-1/3}$, $ \B_{\lambda,\alpha,\beta}^\D$ is invertible and  it holds that
    \begin{equation}
\label{eq:116}
\big\|(\B_{\lambda,\alpha,\beta}^\D)^{-1}\big\|+
      \big\|\frac{d}{dx}\, (\B_{\lambda,\alpha,\beta}^\D)^{-1}\big\|\leq
      C\,\alpha^{-1/2}\beta^{- 5/6} \,.
  \end{equation}
\end{proposition}
\begin{proof}
 Let $\tilde{v}_D$ be given by (\ref{eq:65}c), and consequently satisfy  \eqref{eq:75}.
We decompose $\tilde{v}_D$ in the following manner
\begin{displaymath}
  \tilde{v}_D=\check{v}_D^1+\check{v}_D^2\,,
\end{displaymath}
where 
\begin{displaymath}
\check{v}_D^1=-i\beta(\LL_\beta^D-\beta\lambda)^{-1} U^{\prime\prime}\phi 
\end{displaymath} 
and
\begin{displaymath}
  \check{v}_D^2=(\LL_\beta^D-\beta\lambda)^{-1}[\phi^{\prime\prime}(1)\hat{g}_+ +\phi^{\prime\prime}(-1)\hat{g}_-+f]\,.
\end{displaymath}
Using \cite[Lemma 5.7]{AH1}, with
$g=\phi^{\prime\prime}(1)\hat{g}_+ +\phi^{\prime\prime}(-1)\hat{g}_-+f$ yields
\begin{displaymath}
 \|\check{v}_D^2\|_1\leq C\beta^{-5/6}[|\phi^{\prime\prime}(1)|\,\|\hat{g}_+\|_2
  +|\phi^{\prime\prime}(-1)|\,\|\hat{g}_-\|_2+\|f\|_2]\,.
\end{displaymath}
By  \eqref{eq:53} and  \eqref{eq:115}  we then obtain that
\begin{equation}
  \label{eq:117}
\|\check{v}_D^2\|_1\leq C\big([1+\lambda_m\beta^{1/3}]^{ -3/4} \beta^{-1/3}\log \beta\|\phi\|_\infty +\beta^{-5/6}\|f\|_2\big)  \,.
\end{equation}
Applying \eqref{eq:38} with $f=-i\beta U''\phi$ yields
\begin{displaymath}
  \Big\|\check{v}_D^1+
  \frac{U^{\prime\prime}(x_\nu)\phi(x_\nu)}{U-\nu-i\max(-\mu,\beta^{-1/3})}\Big\|_1\leq
 C\, \|\phi\|_{1,2}\,.
\end{displaymath}
Consequently, using \eqref{eq:117} and Sobolev embeddings, we may write
\begin{subequations}
\label{eq:118}
  \begin{equation}
 \tilde{v}_D =
 -\frac{U^{\prime\prime}(x_\nu)\phi(x_\nu)}{U-\nu-i\max(-\mu,\beta^{-1/3})}+\tilde{g}_d \,,
\end{equation}
where
\begin{equation}
  \|\tilde{g}_d\|_1\leq C\,\big (\|\phi\|_{1,2}+\beta^{-5/6}\|f\|_2\big)
\end{equation}
\end{subequations}

By (\ref{eq:65}c)  and \eqref{eq:118}  we may write
\begin{displaymath}
  -\phi^{\prime\prime}+\alpha^2\phi=-\phi^{\prime\prime}(1)\hat{\psi}_+ -\phi^{\prime\prime}(-1)\hat{\psi}_-
  -\frac{U^{\prime\prime}(x_\nu)\phi(x_\nu)}{U-\nu-i\max(-\mu,\beta^{-1/3})}+\tilde{g}_d \,.
\end{displaymath}
Taking the scalar product with $\phi$ and integrating by parts yields 
\begin{multline}
\label{eq:119}
  \|\phi^\prime\|_2^2+\alpha^2\|\phi\|_2^2 = -\phi^{\prime\prime}(1)\langle \phi,\hat{\psi}_+\rangle
  -\phi^{\prime\prime}(-1)\langle \phi,\hat{\psi}_-\rangle \\
 - \Big\langle \phi,\frac{U^{\prime\prime}(x_\nu)\phi(x_\nu)}{U-\nu-i\max(-\mu,\beta^{-1/3})}\Big\rangle+
  \langle \phi,\tilde{g}_D\rangle\,.  
\end{multline}
To estimate the first term $ -\phi^{\prime\prime}(1)\langle \phi,\hat{\psi}_+\rangle$ on the
right-hand-side of \eqref{eq:119} we use the inequality
\begin{displaymath}
  |\phi(x)| \leq(1-x)^{1/2}\|\phi^\prime\|_2  
\end{displaymath}
to obtain by \cite[Eq. (8.91)]{AH1} 
\begin{displaymath}
  |\langle \phi,\hat{\psi}_+\rangle|\leq\|\phi^\prime\|_2\|(1-x)^{1/2}\hat{\psi}_+\|_1\leq C\beta^{-1/2}[1+|\lambda_+|\beta^{1/3}]^{-3/4}\|\phi^\prime\|_2\,.
\end{displaymath}
Combining the above with \eqref{eq:115} yields 
\begin{displaymath}
|\phi^{\prime\prime}(1)\langle \phi,\hat{\psi}_+\rangle|\leq C\,\big[\beta^{-1/6}\log \beta
[1+|\lambda_+|\beta^{1/3}]^{-1/4}\|\phi^\prime\|_2+\beta^{-1}\|f\|_2\big]\|\phi\|_\infty \,.
\end{displaymath}
A similar estimate can be obtained for  the second term $\phi^{\prime\prime}(-1)\langle
\phi,\hat{\psi}_-\rangle$.\\
 Hence, Sobolev embeddings and the fact that $\phi(1)=0$ yield
\begin{equation}
  \label{eq:120}
 |\phi^{\prime\prime}(1)\langle \phi,\hat{\psi}_+\rangle+\phi^{\prime\prime}(-1)\langle \phi,\hat{\psi}_-\rangle|\leq C\,\big[\beta^{-1/6}\log
 \beta\|\phi^\prime\|_2^{3/2}\|\phi\|_2^{1/2}+\beta^{-1}\|f\|_2\|\phi\|_\infty\big] \,. 
\end{equation}

For the third term on the right-hand-side of \eqref{eq:119} we use
integration by parts to obtain
\begin{displaymath}
    \Big\langle \phi,\frac{U^{\prime\prime}(x_\nu)\phi(x_\nu)}{U-\nu-i\max(-\mu,\beta^{-1/3})}\Big\rangle=
    -\Big\langle \Big(\frac{\phi}{U^\prime}\Big)^\prime,U^{\prime\prime}(x_\nu)\phi(x_\nu)\log
    \big(U-\nu-i\max(-\mu,\beta^{-1/3})\big)\Big\rangle  
\end{displaymath}
Consequently, since
\begin{displaymath}
  \|\log \big(U-\nu-i\max(-\mu,\beta^{-1/3})\big)\|_2\leq C\,,
\end{displaymath}
we can conclude that
\begin{equation}
  \label{eq:121}
  \Big|\Big\langle
  \phi,\frac{U^{\prime\prime}(x_\nu)\phi(x_\nu)}{U-\nu-i\max(-\mu,\beta^{-1/3})}\Big\rangle\Big|\leq
  C\, \|\phi\|_\infty\|\phi^\prime\|_2\leq \hat C\, \|\phi^\prime\|_2^{3/2}\|\phi\|_2^{1/2} \,.
\end{equation}
Finally, for the last term on the right-hand-side it holds, by
(\ref{eq:118}b)
\begin{equation}
\label{eq:122}
    |\langle \phi,\tilde{g}_D\rangle|\leq C\|\phi\|_\infty\|\tilde{g}_D\|_1\leq
    \hat C\, [\|\phi^\prime\|_2^{3/2}\|\phi\|_2^{1/2}+\beta^{-5/6}\|f\|_2\|\phi\|_\infty] \,.
\end{equation}
Substituting \eqref{eq:122}, \eqref{eq:121}, and  \eqref{eq:120} into
\eqref{eq:119} yields
\begin{displaymath}
    \|\phi^\prime\|_2^2+\alpha^2\|\phi\|_2^2 \leq
    C\,\big[\|\phi^\prime\|_2^{3/2}\|\phi\|_2^{1/2}+\beta^{-5/6}\|f\|_2\,\|\phi\|_\infty\big]  \,.
\end{displaymath}
Since by  Young's inequality 
\begin{displaymath}
  \alpha^{1/2}\|\phi^\prime\|_2^{3/2}\|\phi\|_2^{1/2} \leq C\, (\|\phi^\prime\|_2^2+\alpha^2\|\phi\|_2^2) \,,
\end{displaymath}
we obtain that for sufficiently large $\alpha$ we have
\begin{displaymath} 
   \|\phi^\prime\|_2\leq C\alpha^{-1/2}\beta^{-5/6}\, \|f\|_2\,.
\end{displaymath}
\end{proof}

\subsubsection{The case $\alpha\gg\beta^{1/6}$}
For the case where $\alpha\gg\beta^{1/6}$ the estimates of $(\B_{\lambda,\alpha,\beta})^{-1}$
obtained in \cite[\S8.2.2]{AH1} do not depend on the assumption
$|U^{\prime\prime}|>0$, and hence we can rely on the same results in this work
as well. For the convenience of the reader we bring here, without
proof, the statements of Proposition 8.4 and Remark 8.5 in \cite{AH1}.
For that matter we also refer the reader to the definition of $\hat
\mu_m$ (which must be positive) in \cite[Eq. (6.57)]{AH1}.
\begin{proposition}
  \label{lem:no-slip-large-alpha} For any $\varkappa >0$, 
 there exist positive $\beta_0$,  $\alpha_2$,  and $C$ such that for all $\beta\geq \beta_0$ it
 holds that
  \begin{equation}
\label{eq:123}
      \sup_{
        \begin{subarray}{c}
         \Re\lambda\leq\beta^{-1/3}[J_m^{2/3}\hat{\mu}_m-\varkappa-\alpha^2\beta^{-2/3}/2] \\
          \alpha_2\beta^{1/6}\leq \alpha 
        \end{subarray}}\big(\alpha \big\|(\B_{\lambda,\alpha,\beta}^\D)^{-1}\big\|+
      \Big\|\frac{d}{dx}\, (\B_{\lambda,\alpha,\beta}^\D)^{-1}\Big\|\big)\leq
    C\,\beta^{-5/6}\,,
  \end{equation}
 where $J_m=\min (U'(-1),U'(1))$ (see Proposition \ref{lem:boundary}).
\end{proposition}
From the proposition we deduce that  there exist positive
  $\beta_0$,  $\alpha_1$, $\alpha_2$,  $\Upsilon$ and $C$ such that for all $\beta\geq \beta_0$ it 
holds
  \begin{equation}
\label{eq:123a}
      \sup_{
        \begin{subarray}{c}
         \Re\lambda\leq\Upsilon \beta^{-1/3} \\
          \alpha_2\beta^{1/6}\leq \alpha \leq \alpha_1 \beta^{1/3}
        \end{subarray}}\big(\alpha \big\|(\B_{\lambda,\alpha,\beta}^\D)^{-1}\big\|+
      \Big\|\frac{d}{dx}\, (\B_{\lambda,\alpha,\beta}^\D)^{-1}\Big\|\big)\leq
    C\,\beta^{-5/6}\,,
  \end{equation}
  Note that \eqref{eq:123a} is valid for $\beta^{-1/3}\alpha$ which is small
  enough as in the preceding subsection.  Note further that \eqref{eq:123a}
  provides a better estimate for $\|(\B_{\lambda,\alpha,\beta}^\D)^{-1}\|$ but worse
  for $\|d/dx\, (\B_{\lambda,\alpha,\beta}^\D)^{-1}\|$.

\subsection{Large $|\lambda|$}
\label{sec:large-lambda}
\begin{proposition}
\label{lem:large-lambda}
There exist positive $\beta_0$, $\lambda_0$, and $C$ such that, for all $\beta\geq \beta_0$
 $|\lambda|\geq \lambda_0$,  and $\mu<1$,  $\B_{\lambda,\alpha,\beta}^\D $ is invertible  and  it holds that
    \begin{equation}
\label{eq:124}
      \sup_{0\leq\alpha}\big\|(\B_{\lambda,\alpha,\beta}^\D)^{-1}\big\|+
      \big\|\frac{d}{dx}\, (\B_{\lambda,\alpha,\beta}^\D)^{-1}\big\|\leq C \beta^{-1} |\lambda|^{-1}\,.
  \end{equation}
\end{proposition}
\begin{proof}
  The proof is rather straightforward. Since for $|\lambda|\geq \lambda_0$ we have either
  $|\nu|\geq \lambda_0/2$ or $\mu\leq -\lambda_0/2$ (or both). Let $\phi\in D(\B_{\lambda,\alpha,\beta})$ and
  $f\in L^2(-1,1)$ satisfy $\B_{\lambda,\alpha,\beta}\, \phi=f$.\\

{\em Step 1}: Prove \eqref{eq:124} in the case $|\nu|\geq \lambda_0/2$ and $|\mu|\leq\nu$.\\
Integration by parts yields
\begin{displaymath} 
  \Im\langle\phi,\B_{\lambda,\alpha,\beta}\phi\rangle=\beta\nu\big(\|\phi^\prime\|_2^2+\alpha^2\|\phi\|_2^2\big)- \beta\big(\langle U\phi^\prime,\phi^\prime\rangle+\Re\langle U^\prime\phi,\phi^\prime\rangle+\alpha^2\langle U\phi,\phi\rangle+\langle U^{\prime\prime}\phi,\phi\rangle\big)\,.
\end{displaymath}
From this we easily conclude that
\begin{displaymath}
  \|\phi^\prime\|_2^2+\alpha^2\|\phi\|_2^2\leq\frac{1}{\beta|\nu|}\|\phi\|_2\|f\|_2+ \frac{C}{|\nu|}( \|\phi^\prime\|_2^2+[\alpha^2+1]\|\phi\|_2^2)
\end{displaymath}
For sufficiently large $\lambda_0$ we obtain that 
\begin{equation}
\label{eq:131}
   \sup_{0\leq\alpha}\big\|(\B_{\lambda,\alpha,\beta}^\D)^{-1}\big\|+
      \big\|\frac{d}{dx}\, (\B_{\lambda,\alpha,\beta}^\D)^{-1}\big\|\leq C \beta^{-1} |\nu|^{-1}\,.
\end{equation}
From \eqref{eq:131} we conclude \eqref{eq:124} for $|\mu| \leq |\nu|$ and
$|\nu|\geq \lambda_0/2$. 
  
{\em Step 2}: Prove \eqref{eq:124} in the case $\mu \leq -\lambda_0/2$ and
$|\nu|\leq|\mu|$. 

Integration by parts yields
\begin{displaymath}
  -\Re\langle\phi,\B_{\lambda,\alpha,\beta}\phi\rangle=\|\phi^{\prime\prime}\|_2^2+\alpha^2\|\phi^\prime\|_2^2+\beta|\mu|(\|\phi^\prime\|_2^2+\alpha^2\|\phi\|_2^2)+\beta\Im\langle U^\prime\phi,\phi^\prime\rangle\,.
\end{displaymath}
From the above identity we get
\begin{displaymath}
  \|\phi^\prime\|_2^2+\alpha^2\|\phi\|_2^2\leq \frac{1}{\beta|\mu|}\|\phi\|_2\|f\|_2+  \frac{C}{|\mu|}( \|\phi^\prime\|_2^2+\|\phi\|_2^2)
\end{displaymath}
For sufficiently large $\lambda_0$ we obtain from the above and 
Poincar\'e inequalities that
\begin{displaymath}
   \sup_{0\leq\alpha}\big\|(\B_{\lambda,\alpha,\beta}^\D)^{-1}\big\|+
      \big\|\frac{d}{dx}\, (\B_{\lambda,\alpha,\beta}^\D)^{-1}\big\|\leq  C \beta^{-1} |\mu|^{-1}\,,
\end{displaymath}
which leads to \eqref{eq:124} for
  $|\mu| \geq |\nu|$ and $\mu \leq -\lambda_0/2$. 

Combining the two steps completes the proof of \eqref{eq:124}. 
\end{proof}

\section{Proof of Theorem \ref{thm:orr}}
\label{sec:proof}
{\em Step 1:} Proof of the injectivity for $0 \leq \alpha \leq \alpha_0$, $|\lambda|\leq\lambda_0$
and $\mu\leq-\mu_0$.\\

Let $\alpha_0$ be given by Proposition \ref{prop:case-alpha-log}, 
$\lambda_0$ be the same as in the statement of Proposition
\ref{lem:large-lambda} and $\mu_0$ be chosen as in the statement of Proposition
\ref{lem:no-slip-convex-U}. We show in this step that, there exist
positive $\beta_0>0$ such that, for all $\beta\geq \beta_0$ and $\alpha,\lambda$ such that $0\leq \alpha
\leq \alpha_0$, $\Re\lambda<-\mu_0$, and $|\lambda|\leq\lambda_0$, $\B_{\lambda,\alpha,\beta}^\D$ is injective.

Here $\beta_0 \geq \beta_0^{max}$ where $\beta_0^{max}$ is the maximum of the $\beta_0$'s
appearing in the previous statements in Sections \ref{sec:3} and
\ref{sec:4} .

We argue by contradiction.
Suppose that there exists $(\alpha_n,\beta_n, \lambda_n,\phi_n)$ such that $\beta_n \to +\infty$,
$\{\alpha_n,\beta_n\}_{n\geq 1}\subset \R_+^2$, $0 \leq \alpha_n \leq \alpha_0$,  
\begin{displaymath}
  \{\lambda_n\}_{n\geq1}=
\{\mu_n+i\nu_n\}_{n\geq1} \subset \overline{B(0,\lambda_0)}\,,
\end{displaymath}
where $\mu_n \leq -\mu_0$, and
$\{\phi_n\}_{n\geq 1}\subset H^4(-1,1)\cap H^2_0(-1,1)$ satisfies
  \begin{subequations}
\label{eq:125}
  \begin{equation} 
\|\phi_n\|_{1,2}=1
  \end{equation}
  and
  \begin{equation}
    \Big(-\frac{d^2}{dx^2}+i\beta_n[U+i\lambda_n]\Big)
    \Big(-\frac{d^2}{dx^2}+\alpha^2_n\Big)\phi_n+i\beta_nU^{\prime\prime}\phi_n=0 \,.
  \end{equation}
  \end{subequations}
  To avoid the contradiction with $\|\phi_n\|_{1,2}=1$, Proposition
  \ref{lem:large-lambda} implies that $(\mu_n,\nu_n) \in B(0,\lambda_0)$ and
  Proposition \ref{lem:no-slip-convex-U} implies that $\mu_n\leq-\mu_0$
  for sufficiently large $n$.
  Thus, it remains necessary to reach a contradiction for $(\mu_n,\nu_n) \in B(0,\lambda_0)$
  and $\mu_n \leq  -\mu_0$.
Let
\begin{equation}
\label{eq:132}
  v_D^n=\A_{\lambda,\alpha}\phi_n+ (U+i\lambda_n)[\phi_n^{\prime\prime}(1)\hat{\psi}_+ + \phi_n^{\prime\prime}(-1)\hat{\psi}_-] \,.
\end{equation}
Let $\chi\in H^1_0(-1,1)$. Integration by parts yields
\begin{displaymath}
  \langle\chi^\prime,\phi^\prime_n\rangle+
  \alpha_n^2\langle\chi,\phi_n\rangle+\Big\langle\chi,\frac{U^{\prime\prime}}{U+i\lambda_n}\phi_n\Big\rangle=\Big\langle\chi,\frac{v_D^n}{U+i\lambda_n}\Big\rangle
  -\langle\chi,\phi_n^{\prime\prime}(1)\hat{\psi}_+ + \phi_n^{\prime\prime}(-1)\hat{\psi}_-\rangle\,. 
\end{displaymath}
By \eqref{eq:104} it holds that, as $n\to +\infty$,
\begin{displaymath}
\Big|\Big\langle\chi,\frac{v_D^n}{U+i\lambda_n}\Big\rangle\Big|\leq
\frac{1}{|\mu_0|}\|\chi\|_2\|v_D^n\|_2\to0 \,,
\end{displaymath}
Furthermore, since $|\chi(x)|\leq\|\chi^\prime\|_2 \,(1-x)^{1/2}$ it holds that
\begin{displaymath}
  |\langle\chi,\phi_n^{\prime\prime}(1)\hat{\psi}_+\rangle|\leq C|\phi_n^{\prime\prime}(1)|\,\|\chi^\prime\|_2\,\|(1-x)^{1/2}\hat{\psi}_+\|_1
\end{displaymath}
By \eqref{eq:74} 
and \cite[Eq. (8.91)]{AH1}  it holds
that
\begin{displaymath}
    |\langle\chi,\phi_n^{\prime\prime}(1)\hat{\psi}_+\rangle|\leq
    C\beta_n^{-1/6}[1+|(\lambda_n)_+|\,\beta_n^{1/3}]^{-1/4} \|\phi_n\|_{1,2}  
 \end{displaymath}
and since $\|\phi_n\|_{1,2}=1$ we obtain that
\begin{displaymath}
  \langle\chi,\phi_n^{\prime\prime}(1)\hat{\psi}_+\rangle\to0\,.
\end{displaymath}
In a similar manner we obtain that
\begin{displaymath}
   \langle\chi,\phi_n^{\prime\prime}(-1)\hat{\psi}_-\rangle\to0\,.
\end{displaymath}
We can thus conclude that
\begin{equation}
\label{eq:126}
  \langle\chi^\prime,\phi^\prime_n\rangle+ \alpha_n^2\langle\chi,\phi_n\rangle+\Big\langle\chi,\frac{U^{\prime\prime}}{U+i\lambda_n}\phi_n\Big\rangle\to0\,.
\end{equation}

 By \eqref{eq:132} and \eqref{eq:11} it holds that
  \begin{displaymath}
    \phi_n^{\prime\prime}=\phi_n^{\prime\prime}(1)\hat{\psi}_+ + \phi_n^{\prime\prime}(-1)\hat{\psi}_--\alpha^2\phi-\frac{1}{U+i\lambda_n}[U^{\prime\prime}\phi_n-v_D^n]\,.
  \end{displaymath}
Since by \eqref{eq:74} and \cite[Eq. (8.91)]{AH1} it holds for some
$C>0$ for all $n\geq1$ that
\begin{displaymath}
  \|\phi_n^{\prime\prime}(\pm1)\hat{\psi}_\pm\|_1\leq C \,,
\end{displaymath}
we can use \eqref{eq:104} and the fact that that $\mu\leq-\mu_0$ to obtain
that for another $C>0$ and for all $n\in\N$ we have 
\begin{equation}
  \label{eq:133}
\|\phi_n^{\prime\prime}\|_1\leq C
\end{equation}

Since $(\phi_n,\alpha_n,\lambda_n)$ are bounded in $H^1_0(-1,1)\times\R\times\C$, and in addition, by
\eqref{eq:133}, bounded in $W^{2,1}(-1,1)$ we may move to a subsequence
for which \break $\phi_n\to\phi\in H^1_0(-1,1)\cap W^{2.1}(-1,+1) $ is strongly convergent
in $H^1_0$ and \break $(\alpha_n,\lambda_n)\to(\alpha,\lambda)$. By \eqref{eq:126} the limit must
satisfy, for any $\chi\in H_0^1(-1,+1)$,
\begin{displaymath}
    \langle\chi^\prime,\phi^\prime\rangle+ \alpha^2\langle\chi,\phi\rangle+\Big\langle\chi,\frac{U^{\prime\prime}}{U+i\lambda}\phi\Big\rangle=0\,.
\end{displaymath}
Given that $\Re\lambda\leq-\mu_0$ 
we may conclude by standard elliptic estimates
that \break $\phi\in H^2(-1,1)$ and hence
\begin{equation}
\label{eq:134}
  \A_{\lambda,\alpha}\phi=0 \,.
\end{equation}
 Since
  \begin{displaymath}
     \A_{\lambda,\alpha}\phi=\Big(-\frac{d^2}{dx^2}+\alpha^2\Big)(R_\alpha^\D-i\lambda)\phi\,,
  \end{displaymath}
we may conclude from  Theorem \ref{thm:invicid} 
that $\phi\equiv0$. A
contradiction.

{\em Step 2:} Prove \eqref{eq:9} for $0 \leq \alpha \leq \alpha_0$ and all $\lambda\in\C$ such
that $\Re\lambda\leq\Upsilon\beta^{-1/3}$, where $\Upsilon>0$ is the same as in the statement of
Proposition \ref{lem:no-slip-convex-U}.

From the foregoing discussion in the previous step, It follows that
\break $\text{Ker}\,\B_{\lambda,\alpha,\beta}^\D=\{0\}$ for sufficiently large $\beta$ and for all
$0\leq\alpha\leq\alpha_0$ and $\lambda$ satisfying $|\lambda| \leq \lambda_0$ and  $\mu\leq-\mu_0$. Note that since the
Fredholm index of $\B_{\lambda,\alpha,\beta}^\D$ vanishes (see \cite[\S5.3]{AH2}) it
follows that $\B_{\lambda,\alpha,\beta}^\D$ is bijective. 

 We now obtain a bound on the inverse of $\B_{\lambda,\alpha,\beta}^\D$ for $0 \leq \alpha \leq
\alpha_0$, $|\lambda|\leq\lambda_0$ and $\mu\leq-\mu_0$. To this end let $f\in L^2(-1,1)$. By the
surjectivity of $\B_{\lambda,\alpha,\beta}^\D$ there exists $ \phi\in H^4(-1,1)\cap
H^2_0(-1,1)$ satisfying $\B_{\lambda,\alpha,\beta}^\D\phi=f$. Suppose for a
contradiction that $(\B_{\lambda,\alpha,\beta}^\D)^{-1}$ is unbounded. Then, there
must exist $\{ f_n\}_{n=1}^\infty\subset L^2(-1,1)$ satisfying $f_n\to0$ and $(\alpha_n,\beta_n, \lambda_n,\phi_n)$ such that $\beta_n \to +\infty$,
$\{\alpha_n,\beta_n\}_{n\geq 1}\subset \R_+^2$, $0 \leq \alpha_n \leq \alpha_0$,  
\begin{displaymath}
  \{\lambda_n\}_{n\geq1}=
\{\mu_n+i\nu_n\}_{n\geq1} \subset \overline{B(0,\lambda_0)}\,,
\end{displaymath}
where $\mu_n \leq -\mu_0$, and $\{\phi_n\}_{n\geq 1}\subset H^4(-1,1)\cap H^2_0(-1,1)$
satisfying $\B_{\lambda,\alpha,\beta}^D\phi_n=f_n$ and $\|\phi_n\|_{1,2}=1$ (see \cite[Lemma
3.10]{chenetal23}). Let $v_D^n$ be given by \eqref{eq:132}. We may
still apply \eqref{eq:104} to obtain that $\|v_D^n\|_2\to0$. We then
proceed in precisely the same manner as in step 1 to obtain that
$\phi_n\to\phi$ strongly in $H^1_0(-1,1)$ where $\phi\in H^1_0(-1,1)$ satisfies weakly
\eqref{eq:134} and $\|\phi\|_{1,2}=1$. As above it follows that
$\phi\in H^2(-1,1)$, and hence, by Theorem \ref{thm:invicid} we reach a
contradiction.

In view of the foregoing discussion we may conclude that there exist
$\beta_0$ and $M \geq 1$ such that for $\beta \geq \beta_0$
\begin{equation}
\label{eq:127}
  \sup_{
    \begin{subarray}{c}\strut
     | \lambda| \leq \lambda_0\\
      \mu\leq-\mu_0
    \end{subarray}}
\Big(\|(\B_{\lambda,\alpha,\beta}^\D)^{-1}\|+\Big\|\frac{d}{dx}(\B_{\lambda,\alpha,\beta}^\D)^{-1}\Big\|\Big)\leq M \,,
\end{equation}
for all $0\leq\alpha\leq\alpha_0$. Combining \eqref{eq:127} with Propositions
\ref{prop:case-alpha-log} and \ref{lem:large-lambda} yields that
for sufficiently large $\beta$
\begin{displaymath}
  \sup_{\mu\leq-\mu_0}
\Big(\|(\B_{\lambda,\alpha,\beta}^\D)^{-1}\|+\Big\|\frac{d}{dx}(\B_{\lambda,\alpha,\beta}^\D)^{-1}\Big\|\Big)\leq M\,.
\end{displaymath}
Combining the above with (\ref{eq:68}b), \eqref{eq:124}, and the
Phragm\'en-Lindel\"of Theorem (see \cite[\S{} 6]{AH2})
completes the proof of the theorem in the
case $0\leq\alpha\leq\alpha_0$. 

{\em Step 3:}  Prove \eqref{eq:9} for $\alpha \geq \alpha_0$ and all $\lambda\in\C$
such that $\Re\lambda\leq\Upsilon\beta^{-1/3}-\alpha^2\beta^{-1}/2$.

The proof for $\alpha \geq \alpha_0$ follows immediately from
\eqref{eq:116} and  \eqref{eq:123}.  

\def\cprime{$'$}

\end{document}